\documentclass[a4paper,reqno,12pt]{amsart}
\usepackage[utf8]{inputenc}

\usepackage{mathtools,thmtools}
\usepackage{tikz-cd}%

\usepackage{graphicx,color}
\usepackage{bm}
\usepackage{amsmath,amsthm,amstext,amsfonts,amsbsy}%
\usepackage{amssymb}
\usepackage{latexsym}%
\usepackage{todonotes}
\usepackage[margin=3cm]{geometry}
\usepackage{layout}
\usepackage[T1]{fontenc}
\usepackage{physics}
\usepackage{braket}

\usepackage[pagebackref]{hyperref}
\usepackage[capitalize]{cleveref}
\hypersetup{
     colorlinks = true,
     citecolor  = blue,
     linkcolor  = blue, 
     urlcolor   = blue, 
}
\usepackage{bookmark}

\usepackage{fancyhdr}
\usepackage{enumitem}

\theoremstyle{definition}
\newtheorem{theorem}{Theorem}[section]
\newtheorem*{theorem*}{Theorem}
\newtheorem{definition}[theorem]{Definition}
\newtheorem{definition*}{Definition}
\newtheorem{example}[theorem]{Example}
\newtheorem*{example*}{Example}
\newtheorem{proposition}[theorem]{Proposition}
\newtheorem*{proposition*}{Proposition}

\newtheorem*{note*}{Note}

\newtheorem*{notice*}{Notice}

\newtheorem*{lemma*}{Lemma}

\newtheorem*{fact*}{Fact}

\newtheorem*{question*}{Question}
\newtheorem{conjecture}[theorem]{Conjecture}
\newtheorem*{conjecture*}{Conjecture}
\newtheorem{notation}[theorem]{Notation}
\newtheorem*{notation*}{Notation}
\newtheorem{corollary}[theorem]{Corollary}
\newtheorem*{corollary*}{Corollary}
\newtheorem{remark}[theorem]{Remark}
\newtheorem*{remark*}{Remark}

\newtheorem*{condition*}{Condition}

\newtheorem*{convention*}{Convention}

\newtheorem*{observation*}{Observation}

\newcommand{\deq}{\coloneqq}

\newcommand{\opn}[1]{\operatorname{#1}}
\newcommand{\catn}[1]{\mathbf{#1}}

\newcommand{\xto}[1]{\xrightarrow{#1}}

\newcommand{\hookto}{\hookrightarrow}

\newcommand{\simto}{ 
\mathrel{\raisebox{0.13em}{${\sim}$}}
\kern -0.75em \mathrel{\raisebox{-0.11em}{${\scriptstyle \to}$}}  
}

\renewcommand*{\backrefalt}[4]{%
\ifcase #1 %
\or        [Cited on p.#2.]%
\else      [Cited on pp.#2.]%
\fi}

\newcommand{\PD}[1]{[#1]_{\mathrm{PD}}}

\crefname{equation}{}{}
\crefname{conjecture}{Conjecture}{Conjectures}

\usepackage{mathrsfs}
\usepackage{upgreek}
\numberwithin{equation}{section}
\title{The complement of tropical curves in moderate position on
tropical surfaces}
\author[Y. Tsutsui]{Yuki Tsutsui}
\address{Graduate School of Mathematical Sciences,
The University of Tokyo, 3-8-1 Komaba, Meguro-Ku,
Tokyo, 153-8914, Japan}
\email{tyuki@ms.u-tokyo.ac.jp}

\begin{document}
\begin{abstract}
L\'opez de Medrano, Rinc\'on and Shaw defined 
the Chern classes on
tropical manifolds as an extension
of their theory of
the Chern--Schwartz--MacPherson cycles on matroids.
This makes it possible to define
the Riemann--Roch number of tropical Cartier divisors
on compact tropical manifolds.
In this paper,
we introduce the notion of a moderate position,
and discuss a conjecture
that the Riemann--Roch number
$\opn{RR}(X;D)$ of
a tropical submanifold $D$ of codimension $1$
in moderate position on
a compact tropical manifold
$X$ is equal to the topological Euler characteristic of 
the complement $X\setminus D$.
In particular, we prove it and 
its generalization when $\dim X=2$ and
$X$ admits a Delzant face structure.
\end{abstract}

\maketitle

\section{Introduction}
\subsection{Background}
L\'opez de Medrano, Rinc\'on and Shaw defined 
the Chern classes and the Todd classes on
tropical manifolds in \cite{demedrano2023chern}
as an extension
of their theory of
the Chern--Schwartz--MacPherson cycles on matroids
\cite{MR3999674}.
In particular, this makes it possible to define
the Riemann--Roch number $\opn{RR}(X;D)$
for any (sedentarity-0) tropical Cartier divisor $D$ on
a compact (purely) $n$-dimensional tropical manifold $X$;
\begin{align}
\label{equation-intro-rr}
\opn{RR}(X;D)\deq 
\int_{X}\opn{ch}(\mathcal{L}(D))\opn{td}(X)
\end{align}
where $\mathcal{L}(D)$ is the 
tropical line bundle associated with $D$,
$\opn{ch}(\mathcal{L}(D))\deq 
\sum_{i=0}^{\infty}\frac{c_1(\mathcal{L}(D))^{i}}{i!}$
is the Chern character of $\mathcal{L}(D)$
and $\int_X$ is the trace map of $X$.
We can extend the Riemann--Roch number for
tropical cycles of codimension $1$
by the Poincar\'e duality of tropical cohomology
(\cref{definition-rr-number}).
In algebraic geometry, the Riemann--Roch number
$\opn{RR}(X;D)$ of a given Cartier divisor $D$
on a smooth projective variety
corresponds to
the Euler characteristic of the sheaf cohomology
of the invertible sheaf 
associated with $D$ by the Hirzebruch--Riemann--Roch theorem
or the Grothendieck--Riemann--Roch theorem.

In tropical geometry, invertible sheaves,
particularly the structure sheaf
on tropical manifolds
(e.g. \cite[\textsection 1]{MR3330789}) 
do not form sheaves of Abelian groups.
Consequently, there is no \emph{direct}
analog of the Euler characteristic of 
the sheaf cohomology of 
tropical Cartier divisors now.
(A different tropical analog of the Euler characteristic of
line bundles is pursued in \cite{tsutsui2023graded} by using ideas
of the Strominger--Yau--Zaslow conjecture and
microlocal sheaf theory, but this approach is somewhat special.)

The tropical Riemann--Roch theorem for 
compact tropical curves was proved by 
both Gathmann--Kerber \cite{MR2377750}
and Mikhalkin--Zharkov \cite{MR2457739}
independently
as a generalization of the 
Riemann--Roch theorem for finite graphs
established by Baker--Norine \cite{MR2355607}. 
The generalization of these results and
their derivatives have developed in works such as 
\cite{MR3046301,MR4251610,MR4229604,MR4444458,MR4512397}.
However, researchers have not yet established
the tropical Riemann--Roch
theorem for higher-dimensional
compact tropical manifolds
as a direct generalization of the tropical
Riemann--Roch theorem for compact tropical curves
by \cite{MR2377750,MR2457739}.

Regardless of the specifics, it seems that understanding
the tropical geometric meaning of $\opn{RR}(X;D)$ is 
important.
At least, there exists a common interpretation
for a geometric meaning of the Riemann--Roch number
$\opn{RR}(X;0)$ for the trivial divisor.
In \cite[Conjecture 6.13]{demedrano2023chern},
L\'opez de Medrano, Rinc\'on and Shaw conjectured
that $\opn{RR}(X;0)$ for a compact 
tropical manifold $X$ is equal to the 
topological Euler characteristic 
$\chi(X)$ of $X$, 
and they proved it when 
$X$ is a compact tropical surface admitting
a Delzant face structure \cite[Theorem 6.3]{demedrano2023chern}.
They also proved a sufficiently broad range of 
tropical surfaces admitting Delzant face structures
\cite[Corollary 6.11]{demedrano2023chern}.
(The tropical Noether formula
was studied in \cite{shaw2015tropical} previously.
We also note a study of the Noether
formula for tropical complexes of dimension $2$
in \cite[Proposition 1.3]{cartwright2015combinatorial}).
Many researchers expect that the topological Euler 
characteristic of compact tropical manifolds
should be an analog of that of the sheaf cohomology
of the structure sheaves on complete
smooth algebraic varieties.
In fact, the two are highly related 
via a good degeneration of algebraic varieties
\cite[Corollary 2]{MR3961331}.
Besides, the Conjecture 6.13 of 
\cite{demedrano2023chern} is true for 
compact integral affine manifolds
from Klingler's proof of Chern's conjecture for
special affine manifolds \cite{MR3665000}
since the Todd class of integral affine manifold
is trivial.

In this paper, we discuss
a geometrical meaning of $\opn{RR}(X;D)$
when $D$ is in several non-trivial cases.
We mainly consider it when $X$ is
a compact tropical surface, but we expect 
that there exists a generalization for
higher-dimensional compact tropical manifolds.

\subsection{Main results}

Firstly, we recall some elementary properties
of divisors on algebraic varieties. 
To clarify the similarities we wish to investigate,
we will recall them under strong conditions.

Let $D$ be a nonsingular divisor on a nonsingular
algebraic variety $X$, and a morphism  
$\iota\colon D\to X$ the closed embedding.
Then, there exists the following exact sequence
of sheaves:
\begin{align}
     0\to \mathcal{O}_X(-D)\to 
\mathcal{O}_X\to \iota_*\mathcal{O}_D\to 0.
\end{align}
From this, we get the following equation for the Euler
characteristic of the sheaf cohomology:
\begin{align}
\label{equation-anti-effective-divisor}
\chi(X;\mathcal{O}_X(-D))=\chi(X;\mathcal{O}_X)
-\chi(D;\mathcal{O}_D).
\end{align}
From \eqref{equation-anti-effective-divisor},
the conjecture below is very natural and 
seems to be widely
believed among many researchers
before \cite{demedrano2023chern}, if
we mildly ignore a technical condition of 
tropical submanifolds 
\cite[Definition 2.14]{demedrano2023chern}
(see also
\cite[Definition 4.3]{shaw2015tropical})
and their definition of the Todd class of
tropical manifolds.

\begin{conjecture}
\label{conjecture-rr-c-euler}
Let $X$ be a compact
(purely) $n$-dimensional tropical manifold and
$D$ a tropical submanifold of codimension $1$
on $X$ 
\cite[Definition 2.14]{demedrano2023chern}.
Then,
\begin{align}
\opn{RR}(X;-D)=
\chi (X)-\chi(D)=
\chi (H^{\bullet}_c(X\setminus D;\mathbb{R}))
\end{align}
where $H_c^{\bullet}(X\setminus D;\mathbb{R})$ 
is the cohomology with compact support of $X\setminus D$.
\end{conjecture}

\begin{example}
If $\dim X=1$, then 
\cref{conjecture-rr-c-euler} is trivial.
Let $\dim X=2$ and $K_X$ the canonical cycle of
$X$ \cite[Definition 5.8]{MR2275625}.
The adjunction formula for tropical submanifold
of codimension $1$ on $X$
(\cite[Theorem 6]{shaw2015tropical} or 
\cite[Theorem 5.2]{demedrano2023chern}) gives
\begin{align}
\opn{RR}(X;-D)=\frac{\opn{deg}(D.D+K_X.D)}{2}+\opn{RR}(X;0)
=-\chi(D)+\opn{RR}(X;0).
\end{align}
where $K_X.D$ is the intersection of 
$K_X$ and $D$ in the meaning of \cite{shaw2015tropical}.
Therefore, if $X$ admits a Delzant face structure,
then \cref{conjecture-rr-c-euler} is true
by \cite[Theorem 6.3]{demedrano2023chern}.
(In \cref{proposition-cycle-chern}, we will see the compatibility of different definition of
intersection numbers which is needed.)
\end{example}

The main conjecture of this paper is the following 
which gives
the geometric meaning of $\opn{RR}(X;D)$, i.e.,
the dual of $\opn{RR}(X;-D)$.

\begin{conjecture}
\label{conjecture-rr-euler}
Let $D$ be a tropical submanifold of codimension 
$1$ of an $n$-dimensional compact tropical manifold $X$
\cite[Definition 2.14]{demedrano2023chern}.
If $D$ is in \emph{moderate position}
(\cref{definition-permissible-position}) on $X$,
then 
\begin{align}
\opn{RR}(X;D)=\sum_{k=0}^{\infty}\chi(|D^k|)
=\chi(H^{\bullet}(X\setminus D;\mathbb{R}))
\end{align}
where $|D^{k}|$ is the support of the $k$-th power
of $D$ in $X$ (\cref{notation-power-divisor}).
\end{conjecture}
The first equation of \cref{conjecture-rr-euler}
is a certain tropical analog of
\cite[\textsection 20.6. (14)]{MR1335917},
and we will explain it in \cref{example-sum-formula,remark-grothendieck-group}.
We believe that other researchers also expect 
\cref{conjecture-rr-c-euler,conjecture-rr-euler} for 
the reason that,
when $X \deq \mathbb{T}P^n$,
for the tropical hypersurface $D \deq V(F)$
defined by a $d$-degree tropical homogeneous
smooth polynomial $F$, 
\cref{conjecture-rr-c-euler,conjecture-rr-euler}
follow as a consequence of facts
that are classically known.
We will see it in \cref{example-TPn}.
However,
\cref{conjecture-rr-euler} is more non-trivial than
\cref{conjecture-rr-c-euler}.
In fact, many tropical submanifolds are in moderate position,
but we can easily construct of tropical submanifolds
which are not in moderate position.
We also remark that the first equation of
\cref{conjecture-rr-euler} should hold for more general
cases.

\begin{example}
\label{example-permissible-point}
Let $C$ be a compact tropical curve 
and $C_{\mathrm{reg}}$ the set of all
points whose valency are $2$.
A point $p$ on $C$ is in moderate position
if and only if $p\in C_{\mathrm{reg}}$.

From easy calculation, for any 
tropical submanifold $D$ of codimension $1$, i.e.,
finite subsets contained in $C_{\mathrm{reg}}$, we get
\begin{align}
\opn{RR}(C;D)=\sharp (D)+ \chi(C)
=\chi(C\setminus D).
\end{align}
Therefore, \cref{conjecture-rr-euler} is true
for $C$. The first equation above also holds when
$D$ is not in moderate position.
\end{example}
In \cref{proposition-divisor-poincare,remark-iass},
we will see other examples for evidence 
why it seems that \cref{conjecture-rr-euler}
is true. Moreover, we also prove
the second equation of 
\cref{conjecture-rr-euler} when $D$ is 
relatively uniform on $X$ in
\cref{theorem-euler-power}.
One of interesting points of
\cref{conjecture-rr-c-euler,conjecture-rr-euler} is that
the complement of a tropical submanifold in
a tropical manifold is \emph{not} usually
considered an analog of Zariski open subset
of algebraic variety, but related with the
Riemann--Roch number.
For instance, the complement of two points 
on a tropical elliptic curve
$\mathbb{R}/\mathbb{Z}$ is disconnected.
This feature is far from that of the analytification of Zariski open subsets of
algebraic varieties.
(In fact, Mikhalkin defined
a tropical analog of the open subset of
a given polynomial as the complement of
the full graph \cite[\textsection 3.3]{MR2275625}
of a tropical polynomial in
\cite[Remark 3.5 and Example 3.6]{MR2275625}.)
On the other hand, in \cite{MR3498901,MR3968872}, 
interesting studies have been conducted on
the complements of tropical varieties in
$\mathbb{R}^n$ based on motivations
different from our paper.

Incidentally, the author arrived at
\cref{conjecture-rr-euler} as
a derivation from the study in
\cite{tsutsui2023graded} in order to
construct a concrete method of creating
permissible $C^{\infty}$-divisors on
compact tropical manifolds which
satisfies \cite[Conjecture 1.2]{tsutsui2023graded}. 
We will see relationships between 
this paper and \cite{tsutsui2023graded}
in \cref{remark-c-infinity-divisor}.
The relationships which are
remarked in \cref{remark-c-infinity-divisor}
suggest that the cohomology of the complement
of a tropical submanifold of codimension $1$
contains a piece of data
of a homological invariant of the line bundle
associated with it.

One of main theorem in this paper
is the following, and its proof is not difficult:
\begin{theorem}[{Main theorem}]
\label{theorem-rr-euler-surface}
Let $D$ be a tropical submanifold of codimension $1$
in moderate position on a compact tropical surface
$X$. Then,
\begin{align}
\chi(X\setminus D)=\frac{\opn{deg}((D-K_X).D)}{2}+
\chi(X).
\end{align}
\end{theorem}
From \cite[Theorem 6.3]{demedrano2023chern}
and \cref{theorem-rr-euler-surface}, we get
the following corollary:
\begin{corollary}
\label{corollary-ds-euler-rr}
If $X$ is a compact tropical surface admitting 
a Delzant face structure, then
\cref{conjecture-rr-euler} is true.
\end{corollary}
From this point, 
we will also consider an extension of both
\cref{conjecture-rr-c-euler,conjecture-rr-euler}.
We recall well-known properties in algebraic geometry again.
Let $X$ be a nonsingular projective variety,
$D$ a nonsingular divisor on $X$ and
$\iota\colon D\to X$ the inclusion of $D$.
For any divisor $D'$ on $X$, we have
the short exact sequence:
\begin{align}
0 \to \mathcal{O}_X(D'-D)\to \mathcal{O}_X(D')
\to \mathcal{O}_X(D')
\otimes_{\mathcal{O}_X} \iota_*\mathcal{O}_D \to 0. 
\end{align}
From the projection formula for locally free sheaves,
we get 
\begin{align}
\chi(X;\mathcal{O}_X(D')\otimes_{\mathcal{O}_X} \iota_*\mathcal{O}_D)
=\chi(X;\iota_*(\iota^{*}\mathcal{O}_X(D')\otimes_{\mathcal{O}_D} \mathcal{O}_D))
=\chi(D;\iota^{*}\mathcal{O}_X(D')).
\end{align}

If $D'$ is nonsingular and the intersection
$D'\cap D$ is a nonsingular effective divisor
on $D$, then
\begin{align}
\chi(D;\iota^{*}\mathcal{O}_X(D'))=\chi(D;\mathcal{O}_D(D'\cap D)).
\end{align}
It may seem that 
the assumption for a pair $(D,D')$ above is
too strict.
However, through the application of the theorem of
Bertini, we can select a pair $(D,D')$ of divisors 
satisfying the condition above 
and $D_0\sim D'-D$ for a given divisor $D_0$ on $X$.
For such a pair, the following equation holds:
\begin{align}
\chi(X;\mathcal{O}_X(D'-D))=
\chi(X;\mathcal{O}_X(D'))-
\chi(D;\mathcal{O}_D(D'\cap D)).
\end{align}
By combining the observation and \cref{conjecture-rr-euler}, 
we can also expect the following conjecture:
\begin{conjecture}
\label{conjecture-rr-bertini}
Let $X$ be a compact $n$-dimensional
tropical manifold. 
Let $D,D'$ be tropical submanifolds of codimension $1$
on $X$ or empty.
Assume $D$ and $D'$ satisfy the following conditions:
\begin{enumerate}
\item $D'$ is in moderate position on $X$.
\item The restrction $D'|_{D}$ of $D'$ on $D$ is 
a tropical submanifold of $D$ such that its support is   
$D'\cap D$.
\item $D'\cap D$ is in moderate position on $D$.
\end{enumerate}
Then,
\begin{align}
\label{equation-rr-bertini}
\opn{RR}(X;D'-D)=\sum_{k=0}^{\infty}\chi(|D'^k|)-
\sum_{k=0}^{\infty}\chi(|(D'|_{D})^k|)
=&\chi(H^{\bullet}(X\setminus D',D\setminus D';\mathbb{R})) \notag \\
(=&\chi(X\setminus D')-\chi(D\setminus (D'\cap D)))
\end{align}
where $H^{\bullet}(X\setminus D',D\setminus D';\mathbb{R})$
is the relative cohomology of the pair
$(X\setminus D',D\setminus D')$.
\end{conjecture}

\begin{definition}
A pair $(D,D')$ of tropical submanifolds of codimension 
$1$ on a compact tropical manifold $X$ is 
\emph{in moderate position} if $(D,D')$ satisfies
the condition (1)-(3) in \cref{conjecture-rr-bertini}.
\end{definition}
In this paper, we won't discuss how many pairs of 
tropical submanifolds of codimension $1$ in moderate position
exist,
but we expect there exists a sufficient number of them.
Instead, we will see that we can deduce
\cref{conjecture-rr-bertini} from \cref{conjecture-rr-euler}
and a conjecture about the Todd class of tropical 
manifolds (\cref{conjecture-grr-divisor})
in \cref{proposition-euler-to-bertini}.
From this observation, we will
generalize \cref{theorem-rr-euler-surface}
to \cref{theorem-rr-bertini-surface}.
\begin{remark}
If both $D$ and $D'$ is empty, then
\cref{conjecture-rr-bertini} is equivalent to
\cite[Conjecture 6.13]{demedrano2023chern}.
\cref{conjecture-rr-bertini} is equivalent to
\cref{conjecture-rr-c-euler} when 
$D'$ is empty and $D$ is not.
\cref{conjecture-rr-euler} is 
equivalent to
\cref{conjecture-rr-c-euler} when 
$D$ is empty and $D'$ is not.
Therefore, we can consider 
\cref{conjecture-rr-bertini} as a generalization
of the three conjecture above.
Moreover, it is important that 
the RHS of \eqref{equation-rr-bertini} 
can be considered as a certain homological data.
This data is also related with the study in 
\cite{tsutsui2023graded}, so we expect 
\cref{conjecture-rr-bertini} is highly related
with homological mirror symmetry.
\end{remark}

\subsection{Outline of this paper}
In \cref{section-on-rr-euler}, we mainly
discuss \cref{conjecture-rr-euler} and
give a proof of \cref{theorem-rr-euler-surface}.
In \cref{section-on-rr-bertini}, we mainly
discuss \cref{conjecture-rr-bertini,proposition-euler-to-bertini}.
We also give a generalization of
\cref{theorem-rr-euler-surface} in \cref{theorem-rr-bertini-surface}.

\subsection*{Acknowledgement}
We would like to thank Kazushi Ueda for his continuous advice
and encouragement. We are also thankful to Kris Shaw
for answering my questions on
\cite{shaw2015tropical,demedrano2023chern}.
This work was partially supported 
by JSPS KAKENHI Grant Numbers JP21J14529 and JP21K18575.

\section{Tropical submanifolds in moderate position}
\label{section-on-rr-euler}
\subsection{Tropical manifolds}
In this subsection, we recall
the theory of tropical manifolds from
\cite{shaw2011tropical,MR3330789,mikhalkin2018tropical,MR4637248,demedrano2023chern}.
We also recall it from other references if necessary.
We mainly follow the sheaf theoretic approach of
rational polyhedral spaces in \cite{MR4637248}.
For simplicity, we adopt the definition of
tropical manifold in
\cite[Definition 2.3]{demedrano2023chern}.
Every tropical manifold in the sense of
\cite{demedrano2023chern} induce a structure
of a rational polyhedral space naturally,
and it is a tropical manifold in the sense of
the preprint version
\cite[Definition 6.1]{gross2019sheaftheoretic}
of \cite{MR4637248}.
To distinguish between the subtle differences
in the definitions of tropical manifolds
according to different papers,
similar to \cite[\textsection 6]{MR4637248},
we will refer to tropical manifolds in the sense of
\cite[Definition 6.1]{gross2019sheaftheoretic}
as \emph{locally matroidal} rational polyhedral spaces.

\begin{notation}
Throughout this paper, $\mathbb{T}\deq 
\mathbb{R}\cup\{-\infty\}$.
Let $(X,\mathcal{O}_X^{\times})$ be a rational
polyhedral space \cite[Definition 2.2]{MR4637248}.
The \emph{dimension} $\dim X$ of $X$ is 
the homological dimension of locally compact
Hausdorff spaces (e.g.
\cite[Chapter III. Definition 9.4]{MR842190}).
The \emph{local dimension} of $X$ at $x\in X$
will be denoted by $\dim_x X$
\cite[Chapter III. Definition 9.10]{MR842190}.
A rational polyhedral space $(X,\mathcal{O}_X^{\times})$
is \emph{pure dimensional} if $\dim X$ is finite and
$\dim X=\dim_x X$ for any $x\in X$.
These definitions are compatible with
\cite[Definition 7.1.1]{mikhalkin2018tropical}.
We write $X_{\mathrm{reg}}$ for the set of points
in $X$ such that they have an open neighborhood 
which is isomorphic to an
integral affine manifold
\cite[Definition 2.7]{MR4637248}
(see \cite[\textsection 4.1]{MR3894860}).
See also \cite[Definition 3]{MR2181810}
for the sheaf theoretical
definition of integral affine manifolds.
We write $X_{\mathrm{sing}}\deq 
X\setminus X_{\mathrm{reg}}$.

For a given loopless matroid $M$, let
$L_M$ be the \emph{tropical linear space} of
$M$ in the sense of \cite[\textsection 2.2]{MR4637248}.
When we consider $L_M$ as a tropical cycle on a vector
space, we say $L_M$ the matroidal tropical cycle associated
to $M$ like \cite{demedrano2023chern}.
We write $U_{r,n}$ for the uniform of matroid of rank $r$ 
over $[n]\deq \{1,\ldots,n\}$.
\end{notation}

In this paper, a rational polyhedral subspace is meant
in the following sense
(cf. \cite[Definition 2.14]{demedrano2023chern}).

\begin{definition}
Let 
$(X,\mathcal{O}_X^{\times})$ be a rational polyhedral space
and $Y$ a subspace of $X$.
A rational polyhedral space $(Y,\mathcal{O}_Y^{\times})$ is 
a \emph{rational polyhedral subspace} of
$(X,\mathcal{O}_X^{\times})$ if 
for any $x\in Y$, there exists a chart
$\psi \colon U \to V (\subset \mathbb{T}^{n})$
of $X$ \cite[Definition 2.2]{MR4637248}
such that $x\in U$ and the restriction
$\psi|_{U\cap Y}\colon U\cap Y\to 
\psi(U\cap Y)$ is also a chart of $Y$.
\end{definition}
We note that there exists rational polyhedral spaces
$(X,\mathcal{O}_X^{\times})$
and $(Y,\mathcal{O}_Y^{\times})$
such that $X=Y$ and the identity map of $X$ 
is a morphism from $(Y,\mathcal{O}_Y^{\times})$ to
$(X,\mathcal{O}_X^{\times})$ but not an isomorphism.
(We can construct such an example from a tropical analog of
Frobenius morphism.)
Every locally polyhedral set of a rational polyhedral space
\cite[Definition 2.4 (d)]{MR4637248} has a natural 
structure of a rational polyhedral subspace.

Let $(X,\mathcal{O}_X^{\times})$ be a rational polyhedral space
and $\opn{LC}_x X$ the local cone of $X$ at $x$ ($\in X$)
\cite[\textsection 2.2]{MR4637248}.
Follwoing \cite[Definition 7.1.8]{mikhalkin2018tropical} and 
\cite[Definition 2.3]{demedrano2023chern},
an \emph{atlas} of $X$ means a family 
$\{(U_i,\psi_i)\}_{i\in I}$ of charts 
$\psi_i\colon U_i\to V_i (\subset \mathbb{T}^{n_i})$
such that $\bigcup_{i\in I}U_i=X$.

Let $(X,\mathcal{O}_X^{\times})$ be
a rational polyhedral space which is regular at
infinity \cite[\textsection 6.1]{MR4637248}
(see also \cite[Definition 1.2]{MR3330789}
and \cite[Definition 7.2.4 and Corollary 7.2.11]{mikhalkin2018tropical}).
Then, every point $x$ in $X$ has an open neighborhood $U_x$
which is isomorphic to an open subset of 
$\opn{LC}_x X\times \mathbb{T}^{m_x}$ for some
$m_x\in \mathbb{Z}_{\geq 0}$.
Therefore, every rational polyhedral space which
is regular at infinity is paracompact and
locally contractible,
so the singular cohomology
of it is isomorphic to the sheaf cohomology
of the constant sheaf on it, and thus we identify
the two cohomologies.
Moreover,
every nonempty
rational polyhedral space which is regular
at infinity has the sedentarity function
$\opn{sed}_X\colon X\to \mathbb{Z}$ on $X$
\cite[Definition 7.2.6]{mikhalkin2018tropical}
(see also \cite[Definition 2.4]{demedrano2023chern}).

From definition, for any $x\in X$
\begin{align}
\opn{sed}_X(x)+\dim \opn{LC}_x X=\dim_x X.
\end{align}
The sedentarity function
$\opn{sed}_X$ is upper semiconstant
\cite[Definition 7.1.11]{mikhalkin2018tropical} so 
$\opn{sed}_X$ is upper semicontinuous.
In particular, the following subsets are
locally polyhedral
\cite[Proposition 7.1.12]{mikhalkin2018tropical}: 
\begin{align}
X^{[\geq k]}\deq \{p\in X\mid \opn{sed}_X(p)\geq k\},
\quad 
X_{\infty}\deq X^{[\geq 1]}.
\end{align}
The polyhedral subspace $X_{\infty}$ is called
the \emph{boundary} of $X$ (e.g. \cite{demedrano2023chern}),
or the \emph{divisor at infinity} of $X$
(e.g. \cite[Definition 7.2.9]{mikhalkin2018tropical}).
If $X$ is a
tropical toric variety, then $X_{\infty}$ is a
direct analog of toric boundary.
A rational polyhedral space $(X,\mathcal{O}_X^{\times})$
which is regular at infinity is \emph{locally matroidal}
if $\opn{LC}_x X\simeq L_M$ for some loopless matroid $M$
\cite[\textsection 6]{MR4637248}. When a data
$(X,\{\psi_\alpha \colon U_{\alpha} \to 
\Omega_{\alpha}\subset X_{\alpha}\}_{\alpha \in \mathcal{I}})$
is a tropical manifold (in the sense of 
\cite[Definition 2.3]{demedrano2023chern}),
then the associated rational polyhedral space
is locally matroidal.

Let $(X,\mathcal{O}_X^{\times})$ be a rational
polyhedral space and $(Y,\mathcal{O}_Y^{\times})$
a rational polyhedral subspace.
For every $x\in Y$, 
the inclusion map $\iota\colon Y\to X$ induces
an injection 
$\iota_{*,x}\colon \opn{LC}_x Y\to \opn{LC}_x X$
from the local cone $\opn{LC}_x Y$ of $Y$ at $x$ to 
that of $X$ at $x$.
As long as it does not lead to confusion, we will identify
$\opn{LC}_x Y$ with $\iota_{*,x}(\opn{LC}_x Y)$, and
consider $T_x Y$ as a subspace of $T_{x}X$.
We set 
\begin{align}
\opn{codim}(Y/X)\deq \dim X -\dim Y,\quad 
\opn{codim}_x(Y/X)\deq \dim_x X -\dim_x Y.
\end{align}

The rational polyhedral subspace $Y$ of $X$ is
a rational polyhedral space of
\emph{codimension $d$} if $Y$ and $X$ are pure dimensional
and $\opn{codim}(Y/X)=d$.

Let $(X,\mathcal{O}_X^{\times})$ and 
$(Y,\mathcal{O}_Y^{\times})$ be rational
polyhedral spaces which is regular at infinity and
assume that $(Y,\mathcal{O}_Y^{\times})$
is a rational polyhedral subspace of
$(X,\mathcal{O}_X^{\times})$. 
Then, the following equations and inequalities hold
for all $x\in Y$:
\begin{align}
\opn{sed}_X(x)-\opn{sed}_Y(x)=
\opn{codim}_x(Y/X)-\opn{codim}_0(\opn{LC}_x Y/\opn{LC}_xX),
\end{align}
\begin{align}
\opn{codim}_x(Y/X) \geq 
\opn{sed}_X(x)-\opn{sed}_Y(x)\geq 0.
\end{align}
In particular, if $\opn{codim}_x(Y/X)=1$,
then $\opn{sed}_X(x)-\opn{sed}_Y(x)=0$ or $1$.
Following \cite[\textsection 2.5]{demedrano2023chern},
the rational polyhedral subspace $Y$ of $X$ is
\emph{sedentarity-$0$} if
$\opn{sed}_X(x)=\opn{sed}_Y(x)$ for all
$x\in Y$.

An injective morphism $f\colon Y\to X$ of rational polyhedral spaces
is \emph{locally matroidal} if both $X$ and $Y$ are locally matroidal
and for any inclusion $\opn{LC}_x Y\subset \opn{LC}_x X$ of
the local cones at $x$ comes from the inclusion
$L_M\subset L_N$ induced from some two matroids $M,N$ with
the common ground set
(see also \cite[\textsection 3]{MR3041763} 
or \cite[\textsection 2.4]{MR3032930}).  
If $X$ is a tropical manifold and $Y$ is
a tropical submanifold of $X$
\cite[Definition 2.14]{demedrano2023chern},
then the inclusion map $Y\hookto X$ is locally matroidal. 
We note that a codimension $1$ tropical submanifold of a given
tropical manifold $X$ is essentially same with
a locally degree $1$ divisor on $X$
\cite[Definition 4.3]{shaw2015tropical}.
(We thank Kris Shaw for answering our question about this).
As stressed in \cite[Example 2.15]{demedrano2023chern},
there exist examples such that $Y$ is a tropical manifold and
a rational polyhedral subspace of another tropical manifold
$X$, but $Y$ is not a tropical submanifold of $X$.
We can see such examples
in \cite{MR2594592,MR3339531,shaw2015tropical}.
The support of every tropical cycle in a given
rational polyhedral space is
a closed subset of it,
and thus every tropical submanifold of a given
tropical manifold $X$ is a closed subset of $X$.

\subsection{Moderate position}
Let $(X,\mathcal{O}_X^{\times})$ be
a rational polyhedral space.
For $x\in X$, let 
$\opn{lineal}(X,x)$ be the (maximal) lineality space
$\opn{lineal}(\opn{LC}_x X)$ of 
$\opn{LC}_x X (\subset T_x X)$
\cite[\textsection 2.1]{MR4246795}
(see also \cite[\textsection 3]{demedrano2023chern}).

\begin{remark}
The lineality space of the local cone $\opn{LC}_x X$
at a point $x$ in a rational polyhedral space
$(X,\mathcal{O}_X^{\times})$ in the sense of 
\cite[\textsection 2.1]{MR4246795}
is equivalent to
the maximal lineality space of $\opn{LC}_x X$ in
\cite[\textsection 5]{MR3041763}.
We can check about this as follows:
Since every conical rational polyhedral set $P$ has an isomorphism
$P\simeq P/\opn{lineal}(P)\times \opn{lineal}(P)$, we
may assume $\opn{lineal}(P)$ is trivial.
Therefore, to see that the two coincide,
it is sufficient to observe that 
conical rational polyhedral sets
always possess a fan structure.
We can give a proof of it like
that of the existence theorem
of a triangulation of compact convex polyhedron
in $\mathbb{R}^{n}$ (cf. \cite[Theorem 2.11]{MR665919}).
From definition, every conical rational polyhedral set $P$
is a finite union $\bigcup_{i\in I}\sigma_i$ of 
rational polyhedral cones $\sigma_i$.
Besides, we may assume every $\sigma_i$ is strongly convex.
For every $i\in I$, there exists a
complete fan $\Sigma_i$ on $\mathbb{R}^n$
such that $\sigma_i\in\Sigma_i$.
(We can prove it by an application of Sumihiro's
compactification theorem for toric varieties
\cite[Theorem 3]{MR337963}.)
The refinement $\bigwedge_{i\in I}\Sigma_i$
of a family $\{\Sigma_i\}_{i\in I}$ of fans gives
a fan structure of $P$.
\end{remark}

\begin{definition}
\label{definition-permissible-position}
Let $(X,\mathcal{O}_X^{\times})$
be a rational polyhedral space  and 
$(Y,\mathcal{O}_Y^{\times})$ a
rational polyhedral subspace of $X$.
The rational polyhedral subspace
$Y$ is in \emph{moderate position} on $X$ if
for any $x\in Y$
\begin{align}
     \opn{lineal}(Y,x) \subsetneq
 \opn{lineal}(X,x).
\end{align}
\end{definition}
We remark that
the complement $X\setminus Y$ of tropical submanifold
$Y$ in moderate position on $X$
does not usually satisfy the condition of finite type
\cite[Definition 7.1.14 (c)]{mikhalkin2018tropical}. 
\begin{example}
We retain the notation of 
\cref{definition-permissible-position}.
\begin{enumerate}
\item If $Y$ is in moderate position on $X$, then
$Y\cap\{x\in X\mid \opn{lineal}(X,x)=\{0\}\}
=\varnothing$.
\item The rational polyhedral subspace $Y\cap X_{\mathrm{reg}}$ of $X_{\mathrm{reg}}$
is in 
moderate position on $X_{\mathrm{reg}}$.
In particular, $Y$ is always in moderate position
when $X$ is an integral affine manifold.
\end{enumerate}
\end{example}

If $Y$ is in moderate position, then
the inclusion 
$\opn{LC}_x Y\to \opn{LC}_x X$ induces the following 
exact sequences:
\[\begin{tikzcd}
	0 & {\opn{lineal}(Y,x)} & {T_x Y} & {T_x Y/\opn{lineal}(Y,x)} & 0 \\
	0 & {\opn{lineal}(X,x)} & {T_x X} & 
{T_x X/\opn{lineal}(X,x)} & 0
	\arrow[from=1-1, to=1-2]
	\arrow[from=2-1, to=2-2]
	\arrow[from=2-3, to=2-4]
	\arrow[from=1-3, to=1-4]
	\arrow[from=1-4, to=2-4]
	\arrow[from=1-2, to=2-2]
	\arrow[from=1-2, to=1-3]
	\arrow[from=2-2, to=2-3]
	\arrow[from=1-3, to=2-3]
	\arrow[from=2-4, to=2-5]
	\arrow[from=1-4, to=1-5]
\end{tikzcd}.\]

\begin{proposition}
Let $X$ be a tropical manifold and 
$Y$ a tropical submanifold of codimension $1$
in moderate position on $X$.
Then, $Y$ is a sedentarity-0 submanifold of $X$.
\end{proposition}
\begin{proof}
If $Y$ is in moderate position,
then $\opn{LC}_x Y\not \simeq \opn{LC}_x X$ for all 
$x \in Y$.
On the other hand, if $x \in Y$ satisfies 
$\opn{sed}_X(x)-\opn{sed}_Y(x)=1$, then
$\dim \opn{LC}_x Y=\dim \opn{LC}_x X$,
and thus $\opn{LC}_x Y\simeq \opn{LC}_x X$
\cite[Lemma 2.4]{MR3041763}.
Therefore, $\opn{sed}_X(x)-\opn{sed}_Y(x)=0$ when
$Y$ is in moderate position. 
\end{proof}

\begin{proposition}
\label{proposition-divisor-poincare}
Let $X$ be a purely 
$n$-dimensional compact integral
affine manifold and $D$ a tropical submanifold 
of codimension $1$.
Then, \cref{conjecture-rr-c-euler} is 
equivalent to
\cref{conjecture-rr-euler}. In particular,
\cref{conjecture-rr-euler} is true for
all compact integral affine manifolds when
\cite[Conjecture 6.13]{demedrano2023chern} is
true for any compact tropical manifold.
\end{proposition}
\begin{proof}
Since $X$ is an integral affine manifold,
$D$ is always in moderate position in $X$.
Since $\opn{td}(X)=1$ (see \cref{definition-tropical-todd}),
we have
\begin{align}
\label{equation-trivial-serre}
\opn{RR}(X;-D)=(-1)^{\dim X}\opn{RR}(X;D).
\end{align}
Besides, $X\setminus D$ is a topological manifold, 
and thus we have the following equation from 
the Poincar\'e duality:
\begin{align}
\label{equation-pd}
\chi_c(X\setminus D)=(-1)^{\dim X}\chi(X\setminus D)
\end{align}
where $\chi_c(X\setminus D)\deq 
\chi(H^{\bullet}_c(X\setminus D;\mathbb{R}))$.
Therefore, \cref{proposition-divisor-poincare} follows from
\cref{equation-trivial-serre} and
\cref{equation-pd}.
\end{proof}

\begin{remark}[{Integral affine manifolds with singularities}]
\label{remark-iass}
We expect that we can generalize
\cref{conjecture-rr-euler} and 
\cref{proposition-divisor-poincare}
for integral affine manifold with singularities
(e.g. see \cite{MR2213573,MR2181810,MR4347312}).

Normally, integral affine manifolds with singularities
are analogs of complex manifolds whose canonical bundles
are numerically trivial. If $M$ is a compact and
connected complex manifold whose canonical bundle $K_M$ is
numerically trivial, then $c_1(K_M)$ is a torsion,
i.e., $l\cdot c_1(K_M)=0$ for 
some $l\in \mathbb{Z}_{>0}$.
From the Atiyah--Singer index formula and the Serre duality, 
for any divisor $D$ on $M$, we have
\begin{align}
\label{equation-calabi-yau-euler}
\chi(M;\mathcal{O}_M(-D))=(-1)^{\dim M}
\chi(M;\mathcal{O}_{M}(D)).
\end{align}
The equation \cref{equation-calabi-yau-euler} is
similar with
\cref{proposition-divisor-poincare}, and 
one of essential points of the proof of it
is that $X\setminus D$ is a topological manifold.
Therefore, we expect that we can generalize
\cref{conjecture-rr-euler,conjecture-rr-bertini} for 
integral affine manifolds with singularities
and this description is compatible with
tropical contractions from tropical manifolds
to integral affine manifold with singularities
\cite{yamamoto2021tropical}.
\end{remark}

\begin{example}[{\cite[Example 2.11]{demedrano2023chern}}]
\label{example-TPn}
Let $e_i$ the $i$-th coordinate vector of $\mathbb{R}^n$
and $\Delta_{n}$ be the standard $n$-simplex generated by
$e_i$ ($i=1,\ldots,n$) and the origin of $\mathbb{R}^n$.
The normal fan $\Sigma$ of $\Delta_{n}$ induces 
a compactification $X_{\Sigma}$ of $\mathbb{R}^n$.
The space $X_{\Sigma}$ is a typical example of
tropical toric varieties \cite{MR2428356,MR2511632}
and $X_{\Sigma}$ is
isomorphic to the tropical projective
$n$-space $\mathbb{T}P^{n}$ \cite[Example 3.10]{MR2275625}.
From a direct calculation of the \v{C}ech cohomology of
$\mathcal{O}_{X_{\Sigma}}^{\times}$, we get 
$H^{1}(X_{\Sigma};\mathcal{O}_{X_{\Sigma}}^{\times})
\simeq H^{1,1}(X_{\Sigma};\mathbb{Z})\simeq \mathbb{Z}$.
In general, the Picard group of toric 
schemes over semifields is studied in \cite{MR4016643}. 
(See also \cite{Meyer2011} and
\cite[Chapter 3]{mikhalkin2018tropical} for more detail
about tropical toric varieties.)
Let $f$ be a Laurent polynomial on $\mathbb{R}^{n}$
whose Newton polytope is $d\Delta_{n}$.
The closure $\overline{V_{\mathbb{T}}(f)}$ of 
$V_{\mathbb{T}}(f)$ in $X_{\Sigma}$
is can be considered
as an analog of a hypersurface of projective $n$-space of
a homogeneous polynomial of degree $d$
(see also \cite[Definition 3.4.6]{mikhalkin2018tropical}).
The degree of $\overline{V_{\mathbb{T}}(f)}$
on $\mathbb{T}P^n$ is equal to 
that of $f$ and the first Chern class 
of the line bundle
$\mathcal{L}(\overline{V_{\mathbb{T}}(f)})$.
The tropical hypersurface $V_{\mathbb{T}}(f)$ of $f$
is \emph{smooth} (in the sense of
\cite[\textsection 4.5]{MR3287221})
if the regular subdivision of $f$ is unimodular.
If the regular subdivision of $f$ is unimodular,
then $\overline{V_{\mathbb{T}}(f)}$ is a tropical manifold.
Let $d\Delta_n(\mathbb{Z})$ be the set of lattice points
in $d\Delta_n$, and 
$\opn{int}(d\Delta_n)(\mathbb{Z})$ the set of lattice points
in the (relative) interior of $d\Delta_n$.
It is well-known that the complement 
$\mathbb{T}P^{n}\setminus
\overline{V_{\mathbb{T}}(f)}$ is homotopy equivalent
to $d\Delta_n(\mathbb{Z})$ and 
$\overline{V_{\mathbb{T}}(f)}$ is homotopy equivalent
to the $\sharp (\opn{int}(d\Delta_n)(\mathbb{Z}))$-th
bouquet of $(n-1)$-spheres
(e.g. \cite[Proposition 3.1.6]{MR3287221}
or \cite[Proposition 3.4.12]{mikhalkin2018tropical}).
In particular, every connected component
of $\mathbb{T}P^{n}\setminus
\overline{V_{\mathbb{T}}(f)}$ is a locally closed
polyhedron in 
$\mathbb{T}P^{n}$, so it is contractible.
Moreover, the definition of Chern classes of
tropical toric manifolds
is compatible with
that of algebraic toric manifolds
(cf. \cite[Proposition 13.1.2]{MR2810322}).
Therefore, the Todd class of $\mathbb{T}P^{n}$
also has the same representation of that of
algebraic toric manifolds
(see also \cite[Theorem 13.1.6]{MR2810322}). 
Therefore, we have
\begin{align}
\opn{RR}(\mathbb{T}P^{n};\overline{V_{\mathbb{T}}(f)})=
\sharp d\Delta_n(\mathbb{Z})=
\chi(\mathbb{T}P^{n}\setminus
\overline{V_{\mathbb{T}}(f)}).
\end{align}

We can also get the following equation similarly:
\begin{align}
\label{equation-int-polytope}
\opn{RR}(\mathbb{T}P^{n};-\overline{V_{\mathbb{T}}(f)})=
(-1)^{n}\sharp \opn{int}(d\Delta_n(\mathbb{Z}))=
\chi_{c}(\mathbb{T}P^{n}\setminus
\overline{V_{\mathbb{T}}(f)}).
\end{align}
The second equation of \eqref{equation-int-polytope}
follows from that 
$\chi_{c}(\mathbb{T}P^{n}\setminus
\overline{V_{\mathbb{T}}(f)})
=1-\chi (\overline{V_{\mathbb{T}}(f)})$
and $\overline{V_{\mathbb{T}}(f)}$ is homotopy 
equivalent to the 
$\sharp \opn{int}(d\Delta_n(\mathbb{Z}))$-th
bouquet of $(n-1)$-dimensional spheres.
\end{example}

\subsection{Relatively uniform tropical submanifolds}
In this subsection, we define a good class
of tropical submanifolds in moderate position
and prove the second equation of
\cref{conjecture-rr-euler} for this case
in \cref{theorem-euler-power}. 
\begin{definition}
\label{definition-relatively-uniform}
Let $X$ be an $n$-dimensional tropical manifold $X$
and $\iota\colon D\to X$ is the embedding map of
a codimension-1 and sedentarity-0 tropical submanifold
$D$ on $X$.
The tropical submanifold $D$ of $X$
is \emph{relatively uniform} on $X$ if, for any $x\in D$, 
the pushforward morphism
$\iota_{*x}\colon\opn{LC}_{x}D\to \opn{LC}_{x}X$ is isomorphic to
\begin{align}
\label{equation-relatively-uniform-map}
\opn{id}_{L_M}\times i\colon L_{M}\times L_{U_{r,r+1}}\to 
L_{M}\times L_{U_{r+1,r+1}}
\end{align}
where
$M$ is a loopless matroid, $r$ is a positive integer
and $i$ is the inclusion map induced from the inclusion
$U_{r,r+1}\subset U_{r+1,r+1}$. 
\end{definition}

We note that \eqref{equation-relatively-uniform-map}
itself can be represented by a matroidal map
induced from the parallel connection of matroids
(see \cite[Lemma 3.1]{MR4246795}
or \cite[Proposition 3.7]{demedrano2023chern}).
We don't know whether there exists a
tropical submanifold in moderate position but
not relatively uniform, so
\cref{definition-relatively-uniform} may be equivalent
to \cref{definition-permissible-position}.

In the following proposition, we use
\cref{notation-power-divisor} but it is elementary.
\begin{proposition}
Let $X$ be a purely $n$-dimensional tropical manifold
and $D$ a relatively uniform tropical submanifold of
codimension $1$ on $X$.
Then, for any $x\in |D^{j}|$ for some $j\in \mathbb{Z}_{>0}$
there exists the following isomorphism:
\begin{align}
\label{equation-relatively-uniform-cone}
\opn{LC}_x |D^{j}|\simeq 
L_{M_{x}}
\times L_{U_{r-j+1,r+1}}
\end{align}
where $M_x$ is a loopless matroid which is independent of
the choice of $j$.
\end{proposition}
\begin{proof}
By direct calculation, we have 
$|(L_{U_{r,r+1}})^{j}|=L_{U_{r-j+1,r+1}}$
(see \cite[Example 3.9]{MR2591823}).
The proposition is local, so we may assume 
$X=\mathbb{T}^{m}\times L_{M}\times L_{U_{r+1,r+1}}$ and 
$D=\mathbb{T}^{m}\times 
L_{M}\times L_{U_{r,r+1}}$
for some loopless matroid $M$.
Let $\pi_{|D^j|} \colon |D^{j}|\to L_{U_{r-j+1,r+1}}$
be the projection.
Then, $D=\pi_{X}^*L_{U_{r,r+1}}$ and 
$D|_D=\pi_D^{*}(L_{U_{r,r+1}}|_{L_{U_{r,r+1}}})
=\pi_D^{*}L_{U_{r-1,r+1}}=
\mathbb{T}^{m}\times L_{M}\times L_{U_{r-1,r+1}}$.
By repeating this, we obtain
\eqref{equation-relatively-uniform-cone}.
\end{proof}

The following proposition is also elementary,
but important.

\begin{proposition}
\label{proposition-cpt-complement}
Let $X\deq L_{U_{n+1,n+1}}\simeq \mathbb{R}^n$ and
$D\deq L_{U_{n,n+1}}$. 
For $U_i\in \pi_0(X\setminus D)$, we write 
$\overline{U_i}$ the closure of $U_i$ in $X$
and $(X\setminus D)^{\dagger}\deq \bigsqcup_{\pi_0(X\setminus D)}
\overline{U_i}$.
Then, the canonical inclusion
$i\colon X\setminus D\to
(X\setminus D)^{\dagger}$ is
locally aspheric in the sense of
\cite[Chapter V. Corollary 1.3.2]{MR3838359}, and
the composition of $i$ with
the canonical map $\iota\colon
(X\setminus D)^{\dagger}\to X$ is the inclusion 
map $j\colon X\setminus D\to X$. 
In particular,
\begin{align}
Rj_*\mathbb{R}_{X\setminus D}=\iota_!
\mathbb{R}_{(X\setminus D)^{\dagger}}.
\end{align}

\end{proposition}

\begin{remark}
We can generalize 
the above partial compactification
$(X\setminus D)^{\dagger}$ and 
\cref{proposition-cpt-complement} when
$X$ is a tropical manifold and $D$ is a
relatively uniform tropical submanifold of codimension
$1$ on $X$ naturally.
When $X$ and $D$ is compact, then the partial compactification
$(X\setminus D)^{\dagger}$ of $X\setminus D$
is also compact. 
\end{remark}

\begin{theorem}
\label{theorem-euler-power}
Let $X$ be a purely $n$-dimensional compact tropical manifold
and $D$ a relatively uniform tropical submanifold of
codimension $1$ on $X$. Then,
\begin{align}
\label{equation-euler-power}
\chi(X\setminus D)=\sum_{k=0}^{\infty}
\chi(|D^{k}|)=\sum_{k=0}^{n}(k+1)\chi_c(|D^{k}|\setminus
|D^{k+1}|).
\end{align}
\end{theorem}
\begin{proof}
Let $j\colon X\setminus D\to X$ be the inclusion map
of $X\setminus D$.
Then, we have 
\begin{align}
H^{\bullet}(X\setminus D;\mathbb{R}_{X\setminus D})
\simeq \mathbb{H}^{\bullet}(X;Rj_*\mathbb{R}_{X\setminus D}).
\end{align}
From \cref{equation-relatively-uniform-cone}, we also get
\begin{align}
\label{equation-local-index}
Rj_*\mathbb{R}_{X\setminus D}\simeq j_*\mathbb{R}_{X\setminus D},
\quad \chi((j_*\mathbb{R}_{X\setminus D})_x)=
\sum_{k=0}^{\infty}1_{|D^{k}|}(x).
\end{align}
Fix a finite tropical atlas $\mathcal{U}$ for $X$.
From the inclusion-exclusion principle and the Meyer--Vietoris
sequence of sheaves
(e.g. \cite[p.185]{MR842190}), we only need to check
\begin{align}
\label{equation-open-euler-calculus}
\chi(H^{\bullet}_c(U;(j_*\mathbb{R}_{X\setminus D})|_{U}))
=\sum_{k=0}^{\infty}\chi_c(|D^{k}|\cap U).
\end{align}
for any $(U,\phi)\in \mathcal{U}$.
Since $\mathbb{T}^{m}$ is homeomorphic to 
$\mathbb{R}_{\geq 0}^{m}$ via the extended exponential map
$\opn{exp}\colon \mathbb{T}^{m} \to \mathbb{R}_{\geq 0}^{m}$,
we may consider every rational
polyhedral set in some $\mathbb{T}^{m}$ as a
subanalytic set in $\mathbb{R}^{m}$.
From definition, we can choose an open subset $V$ of 
$\mathbb{R}^{m}$ such that $\opn{exp}\circ\phi(U)$ is
a closed subset of $V$.
The restriction of every rational polyhedral set in $\mathbb{T}^{n}$
on $V$ is also subanalytic \cite[Proposition 8.2.2.(iii)]{MR1299726}
on $V$. Let $\iota\deq \opn{exp}\circ \phi$.
Then, $\iota_*(j_*\mathbb{R}_{X\setminus D})|_{U}$
is an $\mathbb{R}$-constructible sheaf on $V$
\cite[Definition 8.4.3]{MR1299726} from \cref{proposition-cpt-complement}.
From \cite[Theorem 9.7.1]{MR1299726} and 
\eqref{equation-local-index}, we obtain
\eqref{equation-open-euler-calculus}.
\end{proof}

\begin{remark}
In the proof of \cref{theorem-euler-power},
we considered $Rj_*\mathbb{R}_{X\setminus D}$
in order to calculate
the Euler characteristic of $X\setminus D$.
This method comes from \emph{Euler calculus}
\cite{MR970076,MR1115569}.
Euler calculus is an effective way to calculate
the Euler characteristic of the sheaf cohomology of
constructible sheaves. We can consider 
the Euler characteristic of compact supports of 
locally closed subsets of $X$
as an analog of signed measure,
and the RHS of \cref{equation-euler-power} as an
integration of the following simple function over it
\cite[(3.4)]{MR1115569}:
\begin{align}
\chi(Rj_*\mathbb{R}_{X\setminus D})(x)\deq 
\chi((Rj_*\mathbb{R}_{X\setminus D})_x) 
\simeq \chi(\varinjlim_{U\ni x} H^{\bullet}((X\setminus D)
\cap U;\mathbb{R})), \quad (x\in X).
\end{align}
The value $\chi(Rj_*\mathbb{R}_{X\setminus D})(x)$ is
called the \emph{local Euler--Poincar\'e index} of 
$Rj_*\mathbb{R}_{X\setminus D}$ at $x$.
Euler calculus is also a useful way to
investigate other constructible sheaves
on tropical manifolds.
As stressed in \cite[Remark 4.8]{MR4540954},
we can interpret the first equation of 
the tropical Poincar\'e--Hopf theorem
\cite[Theorem 4.7]{MR4540954} as 
the Euler integration of the total complex
$\Omega^{\bullet}_{\mathbb{Z},X}$ of the sheaves of
tropical $p$-forms.
\end{remark}

\subsection{For tropical surfaces}

For the proof of \cref{theorem-rr-euler-surface}, we 
use the following well-known 
proposition for Bergman fans: 
the support of the Bergman fan of
a loopless matroid of rank $2$ 
is isomorphic to the support of the Bergman fan of
a uniform matroid $U_{2,n}$ on
$[n]\deq \{1,\ldots,n\}$.
In fact, we can deduce it from that 
the supports of the Bergman fans of
loopless matroids are isomorphic
when the simplification of them are isomorphic.
Shaw classified local cones which comes from
for some tropical surface in
\cite[Corollary 2.4]{shaw2015tropical}.

\begin{proposition}
\label{proposition-self-intersection}
Let $S$ be a compact tropical surface
and $C$ is a tropical submanifold of codimension $1$
in moderate position on $S$. Then, 
\begin{align}
     |C^2|= C_{\mathrm{sing}}\cap S_{\mathrm{reg}}.
\end{align} 
\end{proposition}
\begin{proof}
First, we will see the classification of
the embeddings
$\opn{LC}_x C \subset \opn{LC}_x S$ of local cones.
If $C$ is in moderate position on $S$, then
$\dim \opn{lineal}(S,x)=1,2$ for any $x\in C$.
\begin{enumerate}[align=left,leftmargin=*]
\item Suppose $\dim \opn{lineal}(S,x)=2$.
Then, $\opn{LC}_x S\simeq \mathbb{R}^2$ and
$\opn{LC}_x C$ should be isomorphic to 
$L_{U_{2,3}}$ or $L_{U_{2,2}}$. 
In particular, $\opn{val}_C(x)=2,3$ in this case.
\item We assume $\dim \opn{lineal}(S,x)=1$ and
$x\in C\setminus C_{\infty}$. 
Then, $\dim \opn{lineal}(C,x)=0$.
In particular, $\opn{val}_C(x)\geq 3$.
From \cite[Corollary 2.4]{shaw2015tropical}, 
we may assume $\opn{LC}_x S\simeq L_{U_{2,m}}\times \mathbb{R}
\simeq L_{U_{2,m}\oplus U_{1,1}}$
for some $m(\geq 3)$.
From this description, we can check $\opn{val}_C(x)\leq m$.
In particular, we can consider $T_x C$ as 
a proper vector subspace of $T_x S$.

Besides, we can check 
$\opn{LC}_x C\cap \opn{lineal}(S,x)=\{0\}$
by indirect proof.
In fact, if $\opn{LC}_x C\cap \opn{lineal}(S,x)\ne \{0\}$,
then $L\deq \opn{Im} (T_x C\to T_x S/\opn{lineal}(S,x))$ is a
proper subspace of $T_x S/\opn{lineal}(S,x)$.
On the other hand,
the convex hull of 
$L\cap (\opn{LC}_x S/\opn{lineal}(S,x))$
does not contain any nontrivial vector subspace.
This contradicts with $\dim T_x S\geq \dim T_0(\opn{LC}_x C)\geq 2$.

Since $\opn{LC}_x C\cap \opn{lineal}(S,x)=\{0\}$,
$T_x C\cap \opn{lineal}(S,x)=\{0\}$. Then,
$T_x C\cap \opn{LC}_x S\simeq L_{U_{2,m}}\simeq \opn{LC}_x C$
from the Balancing condition of $\opn{LC}_x C$.
In particular, $\opn{LC}_x C$ is the intersection cycle of
$T_x C$ and $\opn{LC}_x S$ in $T_x S$, so the self-intersection cycle
of $\opn{LC}_x C$ in $\opn{LC}_x S$ is trivial.
\end{enumerate}
From the above classification of the embeddings
of local cones, we obtain the proof. 

\end{proof}

\begin{remark}
We can also prove from \cite[Theorem 4.2]{MR3032930}
or \cite[Theorem 4.11]{shaw2015tropical}
since the intersection pairing in the sense of
\cite{MR4637248} is compatible with that of
\cite{shaw2015tropical,demedrano2023chern}
(see \cref{proposition-two-intersection}). 
\end{remark}

From now on, we will prove \cref{theorem-rr-euler-surface}.
Most of the proof is the same as \cref{theorem-euler-power},
but for simplicity,
we will show it independently of \cref{theorem-euler-power}.

\begin{proof}[{Proof of \cref{theorem-rr-euler-surface}}]
Let $j\colon S\setminus C\to S$ be the inclusion map
of $S\setminus C$.
Then, the following equation holds:
\begin{align}
\chi(S\setminus C)=\chi(Rj_*\mathbb{R}_{S\setminus C})
\deq \chi(\mathbb{H}^{\bullet}(S;Rj_*\mathbb{R}_{S\setminus C})).
\end{align}
If $C$ is in moderate position on $S$, then
the classification of the local cone of $C$ gives 
\begin{align}
(Rj_*\mathbb{R}_{S\setminus C})_x
\simeq
\begin{cases}
\mathbb{R}[0], \text{ if } x\in S\setminus C, \\
\mathbb{R}^3[0], \text{ if } x\in C_{\mathrm{sing}}
\cap S_{\mathrm{reg}}, \\
\mathbb{R}^2[0], \text{ otherwise.}  
\end{cases}  
\end{align}
In particular, $Rj_*\mathbb{R}_{S\setminus C}\simeq
j_*\mathbb{R}_{S\setminus C}$.
Since $j^{-1}j_*=\opn{id}_{S\setminus C}$,  
from \cite[Proposition 2.3.6 (v)]{MR1299726}
there exists the following short exact sequence:
\begin{align}
0\to j_!\mathbb{R}_{S\setminus C} 
\to j_* \mathbb{R}_{S\setminus C}
\to (j_*\mathbb{R}_{S\setminus C})_{C} \to 0.
\end{align}
In particular, we have
\begin{align}
\chi(S\setminus C)=
\chi_c(S\setminus C)+\chi((j_*\mathbb{R}_{S\setminus C})|_C).
\end{align}
From definition of the pushforward of sheaves, 
$(j_*\mathbb{R}_{S\setminus C})|_C$ is locally constant 
on $C_{\opn{reg}}$ and the stalk of it is $\mathbb{R}^2$.
Therefore, we also have
\begin{align}
\chi((j_*\mathbb{R}_{S\setminus C})|_C)=
2\chi_c(C_{\mathrm{reg}})+
3\chi_c(C_{\mathrm{sing}}\cap S_{\mathrm{reg}})+
2\chi_c(C_{\mathrm{sing}}\cap S_{\mathrm{sing}}).
\end{align}
Therefore, 
\begin{align}
\label{equation-euler-calculus}
\chi(S\setminus C)&=
\chi_c(S\setminus C)+
3\chi_c(C_{\mathrm{sing}}
\cap S_{\mathrm{reg}})+
2\chi_c(C\setminus (C_{\mathrm{sing}}
\cap S_{\mathrm{reg}}))\\
&=\chi_c(S\setminus C)+2\chi_c(C)+
\chi_c(C_{\mathrm{sing}}
\cap S_{\mathrm{reg}}) \notag \\
&=\chi(S)+\chi(C)+\chi(|C^2|) 
\label{equation-iterated-euler} \\
&=\frac{\opn{deg}(C. (C-K_S))}{2}+\chi(S). \notag
\end{align}
In the last equation, we use the adjunction formula of tropical
curves \cite[Theorem 4.11]{shaw2015tropical}.
\end{proof}

\begin{proof}[{Proof of \cref{corollary-ds-euler-rr}}]
It follows from
\cref{theorem-rr-euler-surface} and
\cref{proposition-cycle-chern}.
\end{proof}

\begin{remark}
\label{remark-c-infinity-divisor}
In this remark, we discuss relationships
between \cref{conjecture-rr-euler} and 
\cite{tsutsui2023graded} in some cases.
This remark highly depends on
\cite{tsutsui2023graded}, and you can
skip this remark to understand the main theorem of  
this paper. As explained in the introduction of this paper,
the Euler characteristic 
$\chi_c(X\setminus D)$ of
the cohomology of compact support of
the complement of a tropical submanifold $D$ on
a compact tropical manifold $X$ is an analog of  
\cref{equation-anti-effective-divisor}, so 
it is natural to expect that \cref{conjecture-rr-c-euler} holds.
We will explain about algebraic geometrical backgrounds of
the first equation of
\cref{conjecture-rr-euler} in
\cref{example-sum-formula,remark-grothendieck-group}.
However, the author does not know some theorem
in algebraic geometry or Berkovich geometry
which suggests the \emph{second} equation of
\cref{conjecture-rr-euler} directly
as like \cref{conjecture-rr-c-euler},
except toric varieties or special cases like them.
On the other hand, based on the study in \cite{tsutsui2023graded},
we can expect that 
$H^{\bullet}(X\setminus D)$ is highly related with
the theory of Floer cohomology of
Lagrangian submanifolds, the homological mirror symmetry
and the Strominger--Yau--Zaslow conjecture.
Below, we will see about it by using examples. 
For simplicity, let 
$X$ be $\mathbb{R}/\mathbb{Z}$ and $D$ a finite subset of
$X$. Then, the following isomorphisms of graded modules hold:
\begin{align}
H^{\bullet}(X\setminus D;\mathbb{Z}) 
&\simeq \bigoplus_{E\in \pi_0(X\setminus D)} 
H^{\bullet}(E;\mathbb{Z})
\simeq \bigoplus_{E\in \pi_0(X\setminus D)} \mathbb{Z}[0], \\
H_c^{\bullet}(X\setminus D;\mathbb{Z}) 
&\simeq \bigoplus_{E\in \pi_0(X\setminus D)} 
H^{\bullet}_c(E;\mathbb{Z})
\simeq \bigoplus_{E\in \pi_0(X\setminus D)} \mathbb{Z}[-1].
\end{align}

We can also define similar graded modules
for $C^{\infty}$-divisors on $X$.
In \cite{tsutsui2023graded},
the author proposed the following approach
inspired from microlocal sheaf theory and
the Strominger--Yau--Zaslow conjecture.

Let $\mathcal{A}^{0,0}_X$ be the sheaf of
$(0,0)$-superforms on $X$
\cite[Definition 2.24]{MR3903579}.
Then, the following exact sequence exists:
\begin{align}
0\to \mathcal{O}_X^{\times} \to \mathcal{A}^{0,0}_X
\to \mathcal{A}^{0,0}_X/\mathcal{O}_X^{\times} \to 0.
\end{align}

Since $\mathcal{A}^{0,0}_X$ is acyclic,
the connecting homomorphism
\begin{align}
H^{0}(X;\mathcal{A}^{0,0}_X/\mathcal{O}_X^{\times})
\to H^{1}(X;\mathcal{O}_X^{\times}); s\mapsto \mathcal{L}(s).
\end{align}
is surjective. The group 
$\opn{Div}^{\infty}(X)\deq 
H^{0}(X;\mathcal{A}^{0,0}_X/\mathcal{O}_X^{\times})$
is an analog of the group of Cartier divisors
when $X$ is a boundaryless tropical manifold.
In \cite{tsutsui2023graded}, elements 
in $\opn{Div}^{\infty}(X)$ are called 
\emph{$C^{\infty}$-divisors} on $X$.
Let $\mathcal{C}^{0}_X$ the sheaf of continuous functions
on $X$. Then, the group of tropical Cartier divisors on
$X$ and $\opn{Div}^{\infty}(X)$ is in 
$\Gamma (X;\mathcal{C}^{0}_X/\mathcal{O}_X^{\times})$.
We note that the author 
defined the group of $C^{\infty}$-divisor for \emph{any}
tropical manifold in \cite{tsutsui2023graded}.
If $B$ is an integral affine manifold,
then every $C^{\infty}$-divisor on $B$ defines
a Lagrangian section of the standard Lagrangian
torus fibration $\check{f}_B\colon \check{X}(B)\to B$
on $B$, we can consider $C^{\infty}$-divisors
as derivative objects of Lagrangian sections of
Lagrangian torus fibrations.

For a given finite set $D$ on
$X=\mathbb{R}/\mathbb{Z}$,
we can find a $C^{\infty}$-divisor $s$ on $X$
such that $\mathcal{L}(D)=\mathcal{L}(s)$ as follows.
Let $\pi\colon \mathbb{R}\to \mathbb{R}/\mathbb{Z}$
be the canonical universal covering.
The pullback $\pi^* D$ is a principal divisor 
of a convex function $f$ on $\mathbb{R}$ such that 
the set of strictly convex points of $f$ is $\pi^{-1}(D)$
and satisfies a quasi-periodicity.
By smoothing of $f$, we can find a quasi-periodic
convex $C^{\infty}$-function $g$ such that $f(x)=g(x)$ on 
the complement of a sufficiently 
small neighborhood of $\pi^{-1}(D)$. 
This $g$ gives a $C^{\infty}$-divisor $s_{\alpha}$ on
$X$ such that the associated line bundle on $X$ is equal
to $\mathcal{L}(D)$.
Moreover, we can deform $s_{\alpha}$ to another 
linearly equivalent 
$C^{\infty}$-divisor $s_{1}$ on $X$ such that 
the intersection $s_0\cap s_{1}\deq 
L_{0}\cap L_{s_{1}}$ of the associated 
Lagrangian section $L_{s_{1}}$
of $s_{1}$ and the zero section $L_0$ is contained in 
$X\setminus D$ and $\sharp (E\cap L_{0}\cap L_{s_{1}})=1$
for any $E\in \pi_0(X\setminus D)$.
Then, $s_{1}$ is a \emph{permissible} 
$C^{\infty}$-divisor and $\pi^{*}s_{1}$ is the 
principal divisor of a strictly convex
$C^{\infty}$-function.
More generally,
we can construct a map $s\colon [0,1]\to 
\Gamma (X;\mathcal{C}^{0}_X/\mathcal{O}_X^{\times})$ satisfies
\begin{enumerate}
\item $s(0)=D$, $s(1)=s_1$ and 
$s(t)$ is a prepermissible $C^{\infty}$-divisor
for all $t\in (0,1]$,
\item for all $t\in [0,1]$, the line bundle
$\mathcal{L}(s(t))\in \opn{Pic}(X)$ associated with
$s(t)$ is equal,
\item $X\setminus D\supset s_0\cap 
s(t)\supset s_0\cap s(u)$ for all $0<t\leq u \leq 1$,
\item $s_0\cap s(t)$ is homotopy equivalent to $X\setminus D$
for all $t\in (0,1]$,
\item for any $x\in X\setminus D$, there exists $t\in (0,1]$
such that $x\in s_0\cap s(t)$. 
\end{enumerate}

In \cite{tsutsui2023graded}, the author
defined the graded module $\opn{LMD}^{\bullet}(X;s)$
for a permissible $C^{\infty}$-divisor $s$ on
a compact tropical manifold $X$.
By using the theory of \cite{demedrano2023chern},
we can refine the Conjecture 1.2 in 
\cite{tsutsui2023graded}
like this: the Euler characteristic of
$\opn{LMD}^{\bullet}(X;s)$ 
is equal to the Riemann--Roch number of 
$\mathcal{L}(s)$.
(\cite[Conjecture 1.2]{tsutsui2023graded} does not
give an explicit definition of the Riemann--Roch number
since this conjecture appeared before
\cite{demedrano2023chern}.)

If $X=\mathbb{R}/\mathbb{Z}$, then 
$X$ is an integral affine manifold and  
the graded module $\opn{LMD}^{\bullet}(X;s_{1})$
for $s_{1}$ is isomorphic to the Floer complex
associated with the Lagrangian section $L_{s_{1}}$
of $s_{1}$ \cite[Remark 13]{MR1882331} as a graded
module (see also \cite[\textsection 4.4]{tsutsui2023graded}
for more detail about it). In our case, we have 
the following isomorphisms:
\begin{align}
\opn{LMD}^{\bullet}(X;s_{1}) &\simeq 
\bigoplus_{p\in s_0\cap s_{1}}\mathbb{Z}[0]
\simeq H^{\bullet}(X\setminus D;\mathbb{Z}), \\
\opn{LMD}^{\bullet}(X;-s_{1}) &\simeq
\bigoplus_{p\in s_0\cap s_{1}}\mathbb{Z}[-1]
\simeq H^{\bullet}_c(X\setminus D;\mathbb{Z}).
\end{align}

The method presented here is indeed not canonical,
but the isomorphisms are not coincidences. In fact, 
a similar logic works for tropical toric varieties.

Let $P$ be a top dimensional Delzant lattice polytope
in $\mathbb{R}^{n}$ and $X_P$ the tropical toric variety
of $P$. Recall that,
if $f$ is a tropical Laurent polynomial on 
$\mathbb{R}^{n}$ 
whose Newton polynomial is $P$ and the regular
subdivision of $f$ is unimodular, then the closure
$\overline{V(f)}$ of
the hypersurface of $f$ in $X_P$ is a tropical submanifold
of $X_P$ and the set of connected component of 
$X_P\setminus \overline{V(f)}$ is bijective with the set $P(\mathbb{Z})$
of lattice points in $P$.
By using the Maslov dequantization 
of a family of nonnegative valued Laurent polynomial
which converges to $f$, we get a 
$C^{\infty}$-divisor $s_f$ on $X_P$ whose line bundle is
isomorphic to that of $\overline{V(f)}$, 
$s_0\cap s_{f}\subset X_P\setminus \overline{V(f)}$,
and $\pi_0(s_0\cap s_{f})$ is bijective with 
$\pi_0(X_P\setminus \overline{V(f)})$.
Then, we have
\begin{align}
\opn{LMD}^{\bullet}(X_P;s_{f})
&\simeq \bigoplus_{p\in P\cap \mathbb{Z}^{n}}\mathbb{Z}[0]
\simeq \bigoplus_{W\in \pi_0(X_P\setminus \overline{V(f)})}
H^{\bullet}(W;\mathbb{Z}) \notag \\
&\simeq H^{\bullet}(X_P\setminus \overline{V(f)};\mathbb{Z}), \\
\opn{LMD}^{\bullet}(X_P;-s_{f})
&\simeq \bigoplus_{p\in \opn{int}(P)(\mathbb{Z})}
\mathbb{Z}[-\dim X_P]
\simeq \bigoplus_{W\in \pi_0(X_P\setminus \overline{V(f)})}
H^{\bullet}_c(W;\mathbb{Z}) \notag \\
&\simeq H^{\bullet}_c(X_P\setminus \overline{V(f)};\mathbb{Z}).
\end{align}
See also \cite[Appendix D]{tsutsui2023graded}
for more detail about $C^{\infty}$-divisors on
tropical toric manifolds associated with lattice polytopes.

From these examples, it seems that the cohomology
of the complement of tropical submanifolds $D$ of codimension
$1$ on a compact tropical manifold $X$ has a piece of data of
homological invariant of $\mathcal{L}(D)$.
From analogy of Floer cohomology,
the author expects that we can define the cohomology
of $\mathcal{L}(D)$ if we can define a good 
``Floer differential'' for the graded module
$H^{\bullet}(X\setminus D)$ (over some Novikov field).
One of reason why we need some ``Floer differential'' for
$H^{\bullet}(X\setminus D)$ is explained
in \cref{remark-complement-cohomology}.
We also note that Demazure's theorem
\cite{MR284446} is also related
with the current remark
(see also \cite[\textsection 9.1]{MR2810322}).
\end{remark}

\begin{remark}
\label{remark-complement-cohomology}
Let $C$ be a compact tropical curve and $D$
a finite set of $C\setminus C_{\mathrm{sing}}$.
Then, there exist the following relationships
between the rank of cohomology of compact support
of $C\setminus D$ and the Baker--Norine rank
$r(D)$ of $D$ \cite[Definition 1.12]{MR2377750}: 
\begin{align}
\label{equation-bn-rank-negative}
\opn{rank} H_c^{0}(C\setminus D)=r(D)+1=0, \quad
\opn{rank} H_c^{1}(C\setminus D)=r(K_C-D)+1.
\end{align}
The Baker--Norine rank $r(D)$ is an analog of 
the rank of linear systems of divisors on
algebraic varieties.
There is an explanation that the Baker--Norine
rank is truly a good analog of the rank of linear systems
in \cite{MR2448666}.
From \eqref{equation-bn-rank-negative},
we can say the rank of $H_c^{\bullet}(C\setminus D)$
also behave like the cohomology of a line bundle on
an algebraic variety.
On the other hand, the cohomology
$H^{\bullet}(C\setminus D)$ does not necessarily behave
as an analog of the cohomology of a line bundle on
an algebraic variety.

From now on, let $C$ be a compact tropical curve
of genus $2$ like \cite[Figure 1]{MR2457739}.
Let $E$ be a connected component of  
 $C\setminus C_{\mathrm{sing}}$, and
$D$ a nonempty finite subset of $E$. 
Then,
\begin{align}
\label{equation-cohomology-behaviour}
\opn{rank} H^{0}(C\setminus D)=\sharp(D), \quad
\opn{rank} H^{1}(C\setminus D)=1.
\end{align}

On the other hand, $\opn{deg}(K_C-D)=2-\sharp (D)$.
Therefore, when $\sharp (D)>2$, we get
\begin{align}
\label{equation-rank-behaviour}
r(D)+1=\sharp(D)-1,\quad r(K_C-D)+1=0.
\end{align}
In particular, 
$\opn{rank} H^{0}(C\setminus D)>r(D)+1$ in this case.
The equation $\opn{rank} H^{1}(C\setminus D)=1$ is
independent of the degree of $D$.
It is a very different point from behavior of 
the cohomology of invertible sheaves on algebraic varieties.
This is a main reason why we stress 
that we need to define a ``Floer differential'' for 
$H^{\bullet}(C\setminus D)$ in order to capture a
homological data of tropical line bundles 
in \cref{remark-c-infinity-divisor}.

We expect that we can define such an appropriate
differential of $H^{\bullet}(C\setminus D)$
by a combinatorial way as an analog of
the Floer complex of Lagrangian sections on 
trivalent graphs
\cite{auroux2022lagrangian}.
We also expect the homological mirror symmetry
gives a new perspective on the specialization
inequality of the Baker--Norine rank 
\cite[Lemma 2.8]{MR2448666}. 
\end{remark}

\section{The Riemann--Roch number of pairs of
tropical submanifolds}
\label{section-on-rr-bertini}
In this section, we discuss \cref{conjecture-rr-bertini}.
In order to discuss \cref{conjecture-rr-bertini},
we need preparation from the theory of tropical homology
and Chern classes of tropical manifolds.
First, we note 
the theory of Chern classes of tropical manifold
\cite{demedrano2023chern} is new and 
the Chern(--Schwartz--MacPherson) class of
tropical manifolds is not defined from a direct analog
of the theory of vector bundles on algebraic varieties.
Therefore, there is much that is not sufficiently
understood about Chern classes of tropical manifolds.
(The theory of tropical vector bundle is studied
in \cite{MR2961320,MR4646329} or 
\cite[Theorem 1.8]{amini2020hodge}.)
On the other hand, we can use facts
which arise from elementary properties of
multiplicative sequences (or $m$-sequences)
\cite[\textsection 1]{MR1335917}
(see also \cite[\textsection 19]{MR440554}).
From these properties,
we can investigate the Todd classes
of tropical manifolds to some extent.

In this section, we mainly
discuss properties which 
the Todd classes should have if
a natural conjecture for
tropical Chern classes
(\cref{conjecture-grr-divisor}) holds.
We also prove a generalization of 
\cref{theorem-rr-euler-surface}
in \cref{theorem-rr-bertini-surface}.

\subsection{Tropical homology}
We recall
the theory of tropical homology from
\cite{MR3330789,MR3894860,MR4637248}.
In this subsection, let
$Q$ be a subring of $\mathbb{R}$
and $p,q\in \mathbb{Z}$.
Let $(X,\mathcal{O}_X^{\times})$ be
a rational polyhedral space and 
$\Omega_{\mathbb{Z},X}^{p}$
the sheaf of tropical $p$-forms
on $X$ \cite[Definition 2.7]{MR4637248}
(cf. \cite[\textsection 2.4]{MR3330789}).
Let $\upomega_{X}^{\bullet}$ be
the dualizing complex of $X$
(e.g. \cite[Definition 3.1.16]{MR1299726}).
The $(p,q)$-tropical cohomology
$H^{p,q}(X;Q)$ and 
the $(p,q)$-tropical Borel--Moore homology
$H_{p,q}^{\mathrm{BM}}(X;Q)$ of
$(X,\mathcal{O}_X^{\times})$
are the following $Q$-modules
\cite[Definition 4.1 and 4.3]{MR4637248}:
\begin{align}
H^{p,q}(X;Q)&\deq \opn{Hom}_{D^{b}(\mathbb{Z}_X)}(
\mathbb{Z}_X,\Omega_{\mathbb{Z},X}^{p}[q])
\otimes_{\mathbb{Z}} Q, \\
H_{p,q}^{\mathrm{BM}}(X;Q)
&\deq \opn{Hom}_{D^{b}(\mathbb{Z}_X)}
(\Omega_{\mathbb{Z},X}^{p}[q],\upomega_{X}^{\bullet})
\otimes_{\mathbb{Z}}Q.
\end{align}

For comparison with the original
tropical homology \cite{MR3330789,MR3961331},
see also \cite[Remark 2.8 and Theorem 4.20]{MR4637248}.
Let $\Omega_{\mathbb{Z},X}^{\bullet}
\deq \bigoplus_{j\in \mathbb{Z}_{\geq 0}}
\Omega_{\mathbb{Z},X}^{j}[-j]$ 
be the total complex of 
$\Omega_{\mathbb{Z},X}^{p}$ 
(see \cite[Proposition 3.1]{smacka2017differential}).
Since $\Omega_{\mathbb{Z},X}^{\bullet}$ is a sheaf
of trivial and graded-commutative dga, so the
hypercohomology 
$\mathbb{H}^{\bullet}(X;\Omega_{\mathbb{Z},X}^{\bullet})$
of $\Omega_{\mathbb{Z},X}^{\bullet}$ has
a natural graded-commutative ring structure
(see \cite[Remark 21.130]{gortzwedhorn2023}).
Besides, $\Omega_{\mathbb{Z},X}^{\bullet}$ is trivial,
so there exists an isomorphism
$\mathbb{H}^{\bullet}(X;\Omega_{\mathbb{Z},X}^{\bullet})
\simeq \bigoplus_{p,q\in \mathbb{Z}}
H^{p,q}(X;\mathbb{Z})$ as Abelian groups.
If $f\colon X\to Y$ is a morphism of
rational polyhedral spaces, then
$f$ induces the pullback
$f^{*}\colon \mathbb{H}^{\bullet}(Y;\Omega_{\mathbb{Z},Y}^{\bullet})
\to \mathbb{H}^{\bullet}(X;\Omega_{\mathbb{Z},X}^{\bullet})$
\cite[Proposition 4.18]{MR4637248} and $f^{*}$ is
a graded ring homomorphism.
If $f\colon X\to Y$ is a proper morphism, then
$f$ induces the pushforward 
$f_*\colon H^{\opn{BM}}_{p,q}(X;Q)\to 
H^{\opn{BM}}_{p,q}(Y;Q)$ \cite[Definition 4.9]{MR4637248}.

Let $Z_k(X)$ be the group of tropical $k$-cycles
on $X$ \cite[Definition 3.5]{MR4637248}.
The presheaf $U\to Z_k(U)$ on $X$ is a sheaf,
and we write $\mathscr{Z}_k^{X}$ for it
\cite[p.591]{MR4637248}.
For any $k\in \mathbb{Z}_{\geq 0}$, there
exists the \emph{cycle map}
$\opn{cyc}_X \colon Z_k(X)\to 
H^{\mathrm{BM}}_{k,k}(X;\mathbb{Z})$
and $f_*\circ \opn{cyc}_X=\opn{cyc}_Y \circ f_*$
for any proper morphism $f\colon X\to Y$ of rational
polyhedral spaces 
\cite[Definition 5.4 and Corollary 5.8]{MR4637248}
(see also \cite[Definition 4.13]{MR3894860}).
A purely $n$-dimensional rational polyhedral space
$(X,\mathcal{O}_X^{\times})$
\emph{admits a fundamental class} if
the constant function on 
$X$ with value $1$ forms
an $n$-dimensional tropical cycle
$1_{X_{\mathrm{reg}}}$
\cite[\textsection 6.1]{MR4637248},
and the image
$[X]\deq \opn{cyc}_X(1_{X_{\mathrm{reg}}})
\in H_{n,n}^{\mathrm{BM}}(X;\mathbb{Z})$ 
is called the \emph{fundamental class} of $X$
(see also \cite[Definition 4.8]{MR3894860}).
For simplicity, we sometimes write $X$ as
$1_{X_{\mathrm{reg}}}$.
From definition,
$H_{n,n}^{\mathrm{BM}}
(X;\mathbb{Z})\deq \opn{Hom}_{D^{b}(\mathbb{Z}_X)}
(\Omega_{X}^{n}[n],\upomega_{X}^{\bullet})$, so
the multiplicative structure of $\Omega_X^{\bullet}$
and the fundamental class of $X$ induces
the following morphism:
\begin{align}
\eta_{p}^{X}\colon \Omega_X^{n-p}[n]\otimes^{L}_{\mathbb{Z}_X}
\Omega_X^{p} \to \Omega_X^{n}[n] \xto{[X]}
\upomega_X^{\bullet}. 	
\end{align}
The tensor-hom adjunction of $\eta_{p}^{X}$ is
the following \cite[p.627]{MR4637248}:
\begin{align}
\delta_p^{X}\colon \Omega_X^{n-p}[n]\to 
\mathcal{D}_{\mathbb{Z}_X}(\Omega_X^{p})
\deq R\mathcal{H}om_{\mathbb{Z}_X}(\Omega_X^{p},
\upomega_X^{\bullet}).
\end{align}
Of course, $\delta_0^{X}=[X]$.
Besides, the natural isomorphism from

\noindent
$\opn{Hom}_{D^{b}(\mathbb{Z}_X)}(- 
\otimes^{L}_{\mathbb{Z}_X} \Omega_X^{p}[q],
\upomega_X^{\bullet})$ to
$\opn{Hom}_{D^{b}(\mathbb{Z}_X)}(-,
\mathcal{D}_{\mathbb{Z}_X}(\Omega_X^{p}[q]))$
gives the following commutative diagram
for any 
$\alpha \colon \mathbb{Z}_X\to \Omega_X^{n-p}[n-q]$:
\[\begin{tikzcd}
{\opn{Hom}_{D^{b}(\mathbb{Z}_X)}(\Omega_X^{n-p}[n-q]
\otimes^{L}_{\mathbb{Z}_X} \Omega_X^{p}[q]
,\upomega_X^{\bullet})} & 
{\opn{Hom}_{D^{b}(\mathbb{Z}_X)}
(\Omega_X^{n-p}[n-q],\mathcal{D}_{\mathbb{Z}_X}(\Omega_X^{p}[q]))} \\
{\opn{Hom}_{D^{b}(\mathbb{Z}_X)}(\Omega^{p}[q],
\upomega_X^{\bullet})} & {\opn{Hom}_{D^{b}(\mathbb{Z}_X)}(\mathbb{Z}_X,\mathcal{D}_{\mathbb{Z}_X}(\Omega^{p}[q]))}.
	\arrow[from=1-2, to=2-2]
	\arrow[from=2-1, to=2-2]
	\arrow[from=1-1, to=1-2]
	\arrow[from=1-1, to=2-1]
\end{tikzcd}\]
The morphism of each row of
this commutative diagram is an isomorphism.
In particular,
the image of $\eta^X_{p}$ by the homomorphism
on the left column of the commutative diagram
is the cap product $\alpha\frown [X]$
of $\alpha$ and $[X]$ \cite[\textsection 4.6]{MR4637248}.
(For the sake of visibility, we have reversed 
the action as in \cite[\textsection 4.6]{MR4637248}.)
Therefore, we can identify
the homomorphism
\begin{align}
\mathbb{H}^{-q}(\delta_p^{X})
\colon H^{n-p,n-q}(X;\mathbb{Z})
\to \mathbb{H}^{-q}(X;\mathcal{D}_{\mathbb{Z}_X}
(\Omega_X^{p})) \notag
\end{align}
with the cap product homomorphism
\begin{align}
\label{equation-PD-map}
\cdot \frown [X]
\colon H^{n-p,n-q}(X;\mathbb{Z})
\to
H^{\mathrm{BM}}_{p,q}(X;\mathbb{Z}).
\end{align}

In particular, if $X$ is compact,
then the unique morphism 
$a_X\colon X\to \{\mathrm{pt}\}$
to the one-point space $\{\mathrm{pt}\}$
defines the trace map
$\int_X c\deq a_{X*}(c\frown [X])\in 
H_{0,0}^{\mathrm{BM}}(\{\mathrm{pt}\};\mathbb{Z})
\simeq \mathbb{Z}$ for
$c\in H^{\bullet,\bullet}(X;\mathbb{Z})$.
An $n$-dimensional rational polyhedral space
$(X,\mathcal{O}_X^{\times})$
admitting with a fundamental class satisfies
\emph{Poincar\'e--Verdier duality} if 
$\delta_{k}^{X}$ is an isomorphism for all
$k\in \mathbb{Z}_{\geq 0}$
\cite[Definition 6.4]{MR4637248}.
If $(X,\mathcal{O}_X^{\times})$ satisfies 
Poincar\'e--Verdier duality, then
\cref{equation-PD-map} is an isomorphism
\cite[Coroolary 6.9]{MR4637248}.
We can see every tropical manifold
satisfies the Poincar\'e--Verdier duality
from \cite[Proposition 5.5]{MR3894860} and
\cite[Theorem 6.7]{MR4637248}.
The integer valued Poincar\'e duality for tropical manifolds
admitting a global face structure was firstly proved in
\cite[Theorem 5.3]{MR3894860}.
In recent work,
considerations have also been made using
the Poincar\'e--Verdier duality
and the local Poincar\'e duality for rational polyhedral
spaces as indicators of smoothness
(e.g. \cite{MR4626316,amini2021homology,MR4637248}). 

For simplicity, we will use the following notation.
\begin{notation}
Let $(X,\mathcal{O}_X^{\times})$ be
a rational polyhedral space admitting
a fundamental class $[X]$ and satisfying
the Poincar\'e--Verdier duality.
We write $\opn{PD}_X$ for the inverse of
the cap product homomorphism $\cdot \frown [X]$,
and $\PD{Z}\deq \opn{PD}_X(Z)$ for all 
$Z\in H^{\mathrm{BM}}_{p,q}(X;\mathbb{Z})$.
\end{notation}

\subsection{Tropical Cartier divisors}
Let 
$(X,\mathcal{O}_X^{\times})$ be a rational polyhedral space
and $\opn{PAff}_{\mathbb{Z},X}$ the sheaf
of piecewise integral affine linear functions
on $X$
(see \cite[Definition 4.1]{MR3894860}
or \cite[Definition 3.8 and Remark
3.9]{MR4637248}).
From definition, there exist the following
exact sequence of sheaves:
\begin{align}
\label{equation-divisor-exact}
0 \to  \mathcal{O}_X^{\times} 
\to \opn{PAff}_{\mathbb{Z},X} \to 
\opn{PAff}_{\mathbb{Z},X}/\mathcal{O}_X^{\times}
\to 0.
\end{align}
Besides, we use the notations below:
\begin{align}
\opn{Div}(X)^{[0]}\deq H^{0}(X;
\opn{PAff}_{\mathbb{Z},X}/\mathcal{O}_X^{\times}), 
\quad \mathcal{D}iv_X^{[0]}\deq
\opn{PAff}_{\mathbb{Z},X}/\mathcal{O}_X^{\times}.
\end{align}
The connecting homomorphism 
$\delta \colon \opn{Div}(X)^{[0]}\to 
H^{1}(X;\mathcal{O}_X^{\times})$
induced from \eqref{equation-divisor-exact}
is surjective \cite[Proposition 4.6]{MR3894860}.
From now on, for a given $D\in \opn{Div}(X)^{[0]}$,
let $\mathcal{L}(D)$ be the associated
line bundle of $D$
(see \cite[\textsection 3.5]{MR4637248} and 
\cite[\textsection 4.3]{MR2457739}).

\begin{remark}
The definition of rational functions on tropical
manifolds varies depending on authors.
Therefore, the definition of tropical
Cartier divisors
on tropical manifolds also varies.
In \cite{MR3894860,MR4637248},
the previously mentioned group
$\opn{Div}(X)^{[0]}$ is called
the group of tropical Cartier divisors on $X$,
but this is different from the meaning
in \cite{demedrano2023chern}.
By \cite[Proposition 3.27]{shaw2015tropical},
every tropical cycle of codimension $1$ on a tropical manifold
is a tropical Cartier divisor in the sense of
\cite{shaw2015tropical,demedrano2023chern}.
In contrast, there exist codimension $1$
tropical cycle which does not come from any
elements in $\opn{Div}(X)^{[0]}$.
For example,
the point $\{-\infty\}$ in $\mathbb{T}$ does not
come from $\opn{Div}(\mathbb{T})^{[0]}$.

On the other hand, there exist advantages of
the definition of tropical Cartier divisors in the
sense of \cite{MR3894860,MR4637248}.
One of them is that 
every morphism $f\colon X\to Y$ of rational polyhedral
spaces always induces the pullback
$f^{*}\colon \opn{Div}(Y)^{[0]} \to \opn{Div}(X)^{[0]}$,
and the pullback is compatible with that of 
Picard groups $f^{*}\colon \opn{Pic}(Y) \to \opn{Pic}(X)$
\cite[Propoisition 3.15]{MR4637248}.
\end{remark}

Moreover,
there exists a natural pairing
\cite[\textsection 3.4]{MR4637248}
(cf. \cite[Definition 6.5]{MR2591823}):
\begin{align}
\label{equation-divisor-pairing}
\opn{Div}(X)^{[0]}\times Z_{k}(X)\to Z_{k-1}(X);
(D,A) \mapsto D\cdot A.
\end{align}

In particular, the following equation holds
\cite[Proposition 5.12]{MR4637248}:
\begin{align}
\opn{cyc}_X(D\cdot A)=c_1(\mathcal{L}(D))
\frown \opn{cyc}_X(A).
\end{align}

\begin{notation}
Let $A$ be a tropical $k$-cycle on a rational polyhedral space
$(X,\mathcal{O}_X^{\times})$.
We write $[A]\deq \opn{cyc}_X(A)$. We also write
$[A]_{\mathrm{PD}}\deq \opn{PD}_X(\opn{cyc}_X(A))$
when $(X,\mathcal{O}_X^{\times})$ satisfies
the Poincar\'e duality.
\end{notation}

When $(X,\mathcal{O}_X^{\times})$ is 
a purely $n$-dimensional rational polyhedral space
which admits a fundamental class, there exists 
the following homomorphism 
\cite[Definition 4.14]{MR3894860}:
\begin{align}
\label{equation-divisor-map}
\opn{div}_X\colon
\opn{Div}(X)^{[0]}\to Z_{n-1}(X); D\mapsto D\cdot X.  
\end{align}
See also \cite[Theorem 4.15]{MR3894860}.

Let $A$ be a tropical $k$-cycle on $X$
, $|A|$ the support of $A$
\cite[Definition 3.5]{MR4637248},
and $i\colon |A|\to X$
the inclusion of $|A|$.
We note that we can consider $A$ as an
element of $Z_k(|A|)$ and 
the trivial projection
formula holds from definition:
\begin{align}
\label{equation-projection-formula}
i_*(D|_{|A|}\cdot A)=D\cdot A.
\end{align}

\begin{proposition}
\label{proposition-divisor-map-injective}
If $X$ is a purely $n$-dimensional
tropical manifold,
then the homomorphism \eqref{equation-divisor-map}
is injective.
\end{proposition}
\begin{proof}
The divisor map for each open subset of $X$
naturally induces
a sheaf homomorphism
\begin{align}
\opn{div}_X\colon \mathcal{D}iv_X^{[0]}\to 
\mathscr{Z}_{n-1}^{X}.
\end{align}
Therefore, it is enough to check the homomorphism
of sheaves is injective.
Let $\mathcal{B}$ be an open basis of $X$.
Then, the category of $\mathcal{B}$-sheaves
(\cite[p.49-50]{MR2675155})
is equivalent to the category of sheaves on
$X$, so we only need to check 
$\opn{div}_{U}\colon \mathcal{D}iv_X^{[0]}(U)\to Z_{n-1}(U)$
for any $U\in \mathcal{B}$.
Since $X$ is a tropical manifold, we may assume
$U$ is isomorphic to an open subset
of $\opn{LC}_x X\times \mathbb{T}^{n}$
for some $x\in X$.
Moreover, we may assume every piecewise
integer affine linear function on $U$ is the pullback
of a piecewise
integer affine linear function on $\opn{LC}_x X$ 
by the projection $\opn{LC}_x X\times \mathbb{T}^{n}\to \opn{LC}_x X$
locally. If $X$ is boundaryless, then
a generalization of \cref{proposition-divisor-map-injective}
holds
\cite[Theorem 4.5]{MR4246795}.
Therefore, it has been proved.
\end{proof}

\begin{notation}
\label{notation-power-divisor}
From now on, we identify
$\opn{Div}(X)^{[0]}$ with 
$\opn{div}_X(\opn{Div}(X)^{[0]})$
when $X$ is a tropical manifold.
Under this identification,
we can consider every sedentarity-0 tropical
submanifold as an element of $\opn{Div}(X)^{[0]}$.
In particular, from \cref{equation-divisor-pairing},
we can define the \emph{$k$-th power}
$D^{k}$ of $D\in \opn{Div}(X)^{[0]}$ as an element
of $Z_{n-k}(X)$ and satisfies
\begin{align}
\PD{D^{k}}=\PD{D}^{k}=c_1(\mathcal{L}(D))^{k}.
\end{align}
For simplicity, $D^{0}\deq X$.
Additionally, we suppose $D$ is a tropical submanifold on $X$
and the pullback $D|_{D}$ of $D$ is a tropical submanifold on $D$.
For the embedding morphism $\iota_{D}\colon D\to X$,
we obtain the following equation:
\begin{align}
\iota_{|D|*}(D|_{|D|})=D\cdot D.
\end{align}
We can identify the support $|D|_{|D|}|$ with $|D^2|$ 
within the scope where confusion does not arise.
Let $D^{(1)}\deq D$ and $D^{(i+1)}\deq 
D^{(i)}|_{|D^{(i)}|}=D|_{|D^{(i)}|}$ for all
$i\in \mathbb{Z}_{>0}$. 
Moreover, if $D^{(k)}$ $(k=1,\ldots,m)$ is
tropical submanifolds of $|D^{k-1}|$
and $|D^{k}|=|D^{(k)}|$, then
\begin{align}
\iota_{|D^{(k)}|*}(D^{(k+1)})=
\iota_{|D^{(k)}|*}(D^{(k)}|_{|D^{(k)}|})
=D\cdot D^{k}=D^{k+1}.
\end{align}
Thus, we can identify 
$|D^{k+1}|$ with $|D^{(k+1)}|$
within the scope where confusion does not arise.
\end{notation}

\subsection{Chern classes and Todd classes of
tropical manifolds}
In this subsection, we start to discuss
\cref{conjecture-rr-bertini} and give a generalization
of \cref{theorem-rr-euler-surface}.

\begin{notation}
Let $X$ be a purely $n$-dimensional tropical manifold
and $\opn{csm}_{k}(X)$ the $k$-th
Chern--Schwartz--MacPherson cycle of $X$
\cite[Definition 3.4]{demedrano2023chern}.
We note that we can define
the $k$-th Chern--Schwartz--MacPherson
cycle of $\opn{csm}_{k}(X)$ by a sheaf
theoretical approach  
from \cite[Proposition 3.11 and Lemma 3.12]{demedrano2023chern}
and \cite[Lemma 4.13]{MR4637248}
(see also \cite[Lemma 5.4]{shaw2022birational}).
Then, we write the $k$-th Chern class of $X$ as
\begin{align}
c_{k}^{\mathrm{sm}}(X)\deq
\opn{PD}_X\circ \opn{cyc}_{n-k}(\opn{csm}_{n-k}(X))
\in H^{k,k}(X;\mathbb{Z}).
\end{align}
This notation $c_{k}^{\mathrm{sm}}(X)$
is used to avoid confusion with the Chern classes of divisors.
As like this, we write the total Chern class of $X$ as
$c^{\mathrm{sm}}(X)\deq\sum_{k=0}^{n} c_{k}^{\mathrm{sm}}(X)$.
For a given $L\in H^{1,1}(X;\mathbb{Z})$, we write
the \emph{Chern character} of $L$ as follows:
\begin{align}
\opn{ch}(L)\deq \opn{exp}(L)
\deq \sum_{i=0}^{\infty}\frac{L^{i}}{i!}\in 
H^{\bullet,\bullet}(X;\mathbb{ Q}).
\end{align}
\end{notation}

In order to discuss the Todd classes of
tropical manifolds, we recall fundamental properties
of Chern classes of complex manifolds from
\cite{MR1335917,MR1644323,MR2810322}.
Let $M$ be a complex manifold and 
$D$ a nonsingular
divisor on $D$. Let $\mathcal{T}_M$ be the tangent
bundle of $M$.
For an embedding $\iota\colon D\to M$ of
$D$, we have the following
exact sequence: 
\begin{align}
\label{equation-total-adjunction-exact}
0 \to \mathcal{T}_{D}\to \iota^{*}\mathcal{T}_M
\to \iota^{*}\mathcal{O}_M(D)\to 0.
\end{align}
From the axioms of Chern classes,
\eqref{equation-total-adjunction-exact} gives 
the adjunction formula for total Chern classes:
\begin{align}
\label{equation-classical-total-adjunction}
\iota^{*}c(\mathcal{T}_M)
=c(\mathcal{T}_{D})c(\iota^{*}\mathcal{O}_M(D)).
\end{align}
Since $c(\mathcal{O}_M(D))=1+c_1(\mathcal{O}_M(D))$, 
the $k$-th part of \eqref{equation-classical-total-adjunction}
is
\begin{align}
\label{equation-classical-total-adjunction-2}
\iota^{*}c_k(\mathcal{T}_M)
=c_{k}(\mathcal{T}_{D})+
c_{k-1}(\mathcal{T}_{D})c_1(\iota^{*}\mathcal{O}_M(D)).
\end{align}
We can discuss an analog of 
\eqref{equation-classical-total-adjunction},
and expect the following conjecture holds:
\begin{conjecture}
\label{conjecture-grr-divisor}
Let $X$ be a pure dimensional compact tropical manifold
and $D$ a tropical submanifold of $X$ of codimension $1$.
Then, the following equations hold for 
$k\in \mathbb{Z}_{\geq 0}$ and the inclusion 
$\iota\colon D \to X$:
\begin{align}
\label{equation-total-adjunction}
\iota^{*}c^{\mathrm{sm}}(X)&=c^{\mathrm{sm}}(D)
(1+\iota^{*}\PD{D}), \\ 
\iota^{*}c^{\mathrm{sm}}_k(X)&=c^{\mathrm{sm}}_k(D)+
c^{\mathrm{sm}}_{k-1}(D)\iota^{*}\PD{D} \,
(\in H^{k,k}(D;\mathbb{Z})).
\end{align}
\end{conjecture}
We also expect that a version
of \cref{conjecture-grr-divisor} for tropical
Chow groups \cite[Definition 3.30]{shaw2015tropical}
also holds.
When $\dim X=1$, \cref{conjecture-grr-divisor} holds
trivially. When $\dim X=2$,
\cref{conjecture-grr-divisor}
is just another representation of the adjunction formula
of locally degree 1 tropical curves on tropical surfaces
(\cite[Theorem 6]{shaw2015tropical},
\cite[Theorem 5.2]{demedrano2023chern}).
We will see later what 
corollaries appear
when \cref{conjecture-grr-divisor} is true.

From now on, we discuss the Todd classes of
tropical manifolds.
At first, we recall $m$-sequences.
The notion of $m$-sequence is introduced
by Hirzebruch in \cite[\textsection 1]{MR1335917},
but some part of explanation in
\cite[\textsection 19]{MR440554}
by Milnor--Stasheff are more comprehensible,
and thus we also follow from \cite[\textsection 19]{MR440554}.
Let $\opn{Todd}\deq 
(\opn{Todd}_j)_{j\in \mathbb{Z}_{\geq 0}}$ be 
the Todd $m$-sequence \cite[\textsection 1.7]{MR1335917}.
In complex geometry, the $k$-th Chern class
$c_k(M)\deq c_k(\mathcal{T}_M)$ of
a given complex manifold $M$ is defined
as an element of $H^{2k}(M;\mathbb{Z})$.
Let $c(M)\deq \sum_{i=0}^{\infty}c_i(M)$ be
the total Chern class of $M$
in the real valued
even cohomology ring $H^{\mathrm{even}}(X;\mathbb{R})$.
Since $H^{\mathrm{even}}(M;\mathbb{R})$
is commutative and a graded $\mathbb{R}$-algebra
and $c_0(M)=1$,
so the total Chern class of $M$ defines 
the Todd class $\opn{td}(M)\deq \sum_{j=0}^{\infty}
\opn{Todd}_j(c_{1}(M),\ldots,c_{j}(M))$ of $M$
\cite[\textsection 10]{MR1335917}.
 
We can define the Todd class
$\opn{td}(X)$ of a given tropical manifold
$X$ by replacing 
$H^{\mathrm{even}}(M;\mathbb{R})$ with
$\bigoplus_{i=0}^{\infty} H^{i,i}(X;\mathbb{R})$
(\cref{definition-tropical-todd}).
In \cite[Conjecture 6.13]{demedrano2023chern},
the Todd class of tropical manifold is
defined as a tropical cycle.
In this paper, we use the Poincar\'e dual of the image 
of the cycle map of it
in order to get closer to the notation 
in algebraic geometry.

\begin{definition}[{\cite[\textsection 6]{demedrano2023chern}}]
\label{definition-tropical-todd}
Let $X$ be a purely $n$-dimensional tropical manifold.
The $j$-th \emph{Todd class} of $X$ is
a cycle in $H^{j,j}(X;\mathbb{R})$
defined from 
the Todd $m$-sequence 
$\opn{Todd}\deq (\opn{Todd}_j)_{j\in \mathbb{Z}_{\geq 0}}$
with respect to $\bigoplus_{i=0}^{\infty}
H^{i,i}(X;\mathbb{R})$ as follows:
\begin{align}
\opn{td}_j(X)\deq \opn{Todd}_j(c_{1}^{\mathrm{sm}}(X),
\ldots,c_{j}^{\mathrm{sm}}(X)),
\end{align}
\begin{align}
\opn{td}(X)\deq
\opn{Todd}(c^{\mathrm{sm}}(X))\deq
\sum_{j=0}^{\infty}\opn{td}_j(X).
\end{align}

\end{definition}

\begin{definition}[{\cite[\textsection 12.1.(2)]{MR1335917}}]
\label{definition-rr-number}
Let $X$ be a purely $n$-dimensional
compact tropical manifold.
The \emph{Riemann--Roch number}
of $D(\in Z_{n-1}(X))$ is the following number:
\begin{align}
\label{equation-rr-number}
\opn{RR}(X;D)\deq \int_X \opn{ch}(\PD{D})\opn{td}(X).
\end{align}
\end{definition}
If $D\in \opn{Div}(X)^{[0]}$,
then we can write \cref{equation-rr-number}
as \cref{equation-intro-rr}
from \cite[Proposition 5.12]{MR4637248}.

Following the theory of algebraic geometry,
we define the \emph{generalized Gysin map}
associated with proper morphism between
tropical manifolds (see  
\cite[Chapter 13. Appendix]{MR2810322}
for classical cases).
We note that the generalized Gysin map
for tropical manifolds
have already appeared in \cite{amini2020hodge}.
\begin{definition}[{cf. \cite{amini2020hodge}}]
Let $X$ and $Y$ be pure dimensional tropical manifolds
and $Q$ a subring of $\mathbb{R}$.
Let $f\colon X\to Y$ be a proper morphism.
The \emph{generalized Gysin map} of $f$ is 
the morphism $f_!\colon H^{\bullet,\bullet}(X;Q)\to 
H^{\bullet,\bullet}(Y;Q)$
which commutes with the following diagram:
\begin{equation}
\begin{tikzcd}
	{H^{\bullet,\bullet}(X;Q)} & {H^{\bullet,\bullet}(Y;Q)} \\
	{H_{\bullet,\bullet}^{\mathrm{BM}}(X;Q)} & {H_{\bullet,\bullet}^{\mathrm{BM}}(Y;Q)}
	\arrow["{\cdot \frown [X]}"', from=1-1, to=2-1]
	\arrow["{\cdot\frown[Y]}", from=1-2, to=2-2]
	\arrow["{f_*}"', from=2-1, to=2-2]
	\arrow["{f_!}", from=1-1, to=1-2]
\end{tikzcd}.   
\end{equation}

\end{definition}
From definition, there exists the following projection formula.
\begin{proposition}
\label{equation-gysin-projection-formula}
Let $f\colon X\to Y$ be a proper morphism
between pure dimensional
tropical manifolds $X$ and $Y$.
Then, for any $c\in H^{\bullet,\bullet}(Y;Q) $ and
$d\in H^{\bullet,\bullet}(X;Q)$ the following equation holds:
\begin{align}
    f_!(f^{*}(c)\cdot d)=c\cdot f_!(d).
\end{align}
In particular, $f_!(f^{*}(c))=c\cdot f_!(1)=c\cdot \PD{X}$
when $X$ is a rational polyhedral subspace of $Y$ and 
$f$ is the inclusion map of $X$.
\end{proposition}

\begin{proof}
Since $\cdot \frown [Y]$ is an isomorphism, we only need
to prove the following:
\begin{align}
    f_!(f^{*}(c)\cdot d)\frown [Y]=(c\cdot f_!(d))\frown [Y].
\end{align}
We can get the projection formula for generalized Gysin
map from that for the pushforward of tropical 
Borel--Moore homologies as follows:
\begin{align}
(c\cdot f_!(d))\frown [Y]&=
c\frown(f_!(d)\frown [Y])=
c\frown f_*(d\frown [X])\\
&=f_*(f^*(c)\frown (d\frown [X]))
=f_*((f^*(c)\cdot d) \frown [X]) \\
&=f_!(f^{*}(c)\cdot d)\frown [Y].
\end{align}
\end{proof}

\begin{example}
\label{example-grr-1}
Let $X$ be a purely $n$-dimensional tropical manifold
and $D$ a tropical submanifold of codimension $1$ on 
$X$. If \cref{conjecture-grr-divisor} is true, then
the generalized Gysin map of the inclusion map
$\iota\colon D\to X$ induces
\begin{align}
c^{\mathrm{sm}}(X)\cdot \PD{D}=\iota_!c^{\mathrm{sm}}(D)
(1+\PD{D}).
\end{align}
The first degree part of the equation above is
\begin{align}
\label{equation-adjunction-dmrs}
c^{\mathrm{sm}}_{1}(X)\cdot \PD{D}=\iota_!c^{\mathrm{sm}}_1(D)
+\PD{D}\cdot \PD{D}.
\end{align}
We may consider \cref{equation-adjunction-dmrs} as 
the Poincar\'e dual of the adjunction formula proved in
\cite[Theorem 5.2]{demedrano2023chern}.
By using the canonical divisor of tropical manifolds,
\eqref{equation-adjunction-dmrs} can be rewritten as follows:
\begin{align}
\iota_! \PD{K_D}=(\PD{K_X}+\PD{D})\PD{D}.
\end{align}
If $\dim X=2$, then this equation also gives
the adjunction formula of tropical surfaces which is proved
in \cite[Theorem 4.11]{shaw2015tropical}
since $H^{2,2}(X;\mathbb{Z})\simeq \mathbb{Z}$.
\end{example}

\begin{example}[{cf. \cite[Chapter 13 Appendix]{MR2810322}}]
\label{example-sum-formula}
We retain the notation
in \cref{example-grr-1}.
We also assume \cref{conjecture-grr-divisor} is true again.
From \eqref{equation-total-adjunction}
and the property of multiplicative sequences, 
\begin{align}
\label{equation-todd-adjunction-1}
\iota^{*}\opn{td}(X)=
\opn{td}(D)\frac{\iota^{*}\PD{D}}
{1-\opn{exp}(-\iota^{*}\PD{D})}.
\end{align}
From \cref{equation-todd-adjunction-1} we also get
\begin{align}
\opn{td}(D)
=\iota^{*}\left(
\frac{1-\opn{exp}(-\PD{D})}{
\PD{D}}\opn{td}(X)\right).
\end{align}
Moreover, the projection formula for 
$\iota$ gives
\begin{align}
\label{equation-grr-divisor-2}
\iota_!(\opn{td}(D))=(1-\opn{ch}(-\PD{D}))\opn{td}(X).
\end{align}
The equation \eqref{equation-grr-divisor-2} 
is an analog of 
the Grothendieck--Riemann--Roch theorem
for a special case.

Let $D'$ be a tropical $(n-1)$-cycle on $X$.
Then, the projection formula for 
the embedding $\iota\colon D\to X$ gives
\begin{align}
\iota_!(\opn{ch}(\iota^{*}(\PD{D'}))\opn{td}(D))
&=\opn{ch}(\PD{D'})\iota_!(\opn{td}(D)) \notag \\
&=(\opn{ch}(\PD{D'})-\opn{ch}(\PD{D'-D}))\opn{td}(X).
\end{align}
If $D'$ is in $\opn{Div}^{[0]}(X)$, then
the pullback $\iota^{*}D'=D'|_{D}$ of $D'$ satisfies
the following:
\begin{align}
\label{equation-rr-number-divisor}
\opn{RR}(X;D'-D)=\opn{RR}(X;D')-
\opn{RR}(D;D'|_{D}).
\end{align}

Similarity to the case of algebraic varieties,
we can define the
\emph{virtual T-genus} of $D_1,\ldots,D_m\in Z_{n-1}(X)$
as follows \cite[\textsection 11.2]{MR1335917}:
\begin{align}
T(X)\deq \opn{RR}(X;0),\quad  T(X;D_1,\ldots,D_m)\deq
\int_X\prod_{j=1}^{m}(1-\opn{exp}(-\PD{D_j}))\opn{td}(X).
\end{align}
We can see $T(X;D)=\opn{RR}(D;0)$ when
\cref{conjecture-grr-divisor} is true.

As explained in \cite[\textsection 20.6]{MR1335917},
the following equation holds
\begin{align}
\opn{exp}(\PD{D_1})=(1-(1-\opn{exp}(-\PD{D_1})))^{-1}
=\sum_{j=0}^{n}(1-\opn{exp}(-\PD{D_1}))^j.
\end{align}

From this, we obtain the following equation
\cite[\textsection 20.6.(14)]{MR1335917}:
\begin{align}
\opn{RR}(X;D_1)=\sum_{j=0}^{n}T(X;\overbrace{D_1,\ldots,D_1}^{j}).
\end{align}

Additionally, we assume every $D_i$ ($i=1,\ldots,m$) is
in $\opn{Div}^{[0]}(X)$ and $D_1=D$.
Then, from 
\cref{equation-rr-number-divisor} and
\cref{equation-gysin-projection-formula},
we also obtain the following equations \cite[Theorem 11.2.1]{MR1335917}:
\begin{align}
T(X;D_1,\ldots,D_m)&=T(D_1;D_2|_{D_1},\ldots,D_m|_{D_1}), \\
\sum_{j=1}^{n}T(X;\overbrace{D,\ldots,D}^{j})
&=\opn{RR}(D;0)+\sum_{j=1}^{n-1}T(D;\overbrace{D|_D,\ldots,D|_{D}}^{j})
=\opn{RR}(X;D|_D).
\label{equation-rr-reduction}
\end{align}
If $D|_D$ is also a tropical submanifold,
then we can repeat this process.
\end{example}

\begin{remark}
\label{remark-grothendieck-group}
In this remark, we see another explanation
of an algebraic geometrical meaning of the first
equation of \cref{conjecture-rr-euler}.
Let $a_X\colon X\to \opn{Spec}k$
be a complete nonsingular algebraic variety over $k$
and $K_0(X)$ the Grothendieck ring of $X$
(see, e.g., \cite[\textsection 15.1]{MR1644323}).
For a coherent sheaf $\mathcal{F}$ on $X$, 
let $[\mathcal{F}]$ be the isomorphism class of
$\mathcal{F}$ in $K_0(X)$.
The addition of $K_0(X)$ is given from
the direct sum of coherent sheaves and
the multiplication of $K_0(X)$ is given
from the derived tensor product of them.   
The unit of $K_0(X)$ is $[\mathcal{O}_X]$
where $\mathcal{O}_X$ is the structure sheaf of $X$.
Since $a_X$ is a proper morphism, so $a_X$
induces the pushforward group homomorphism
$a_{X*}\colon K_0(X)\to K_0(\opn{Spec}k)$,
and $a_{X*}([\mathcal{F}])=
\chi(X;\mathcal{F})$.
For simplicity, we write $\chi\deq a_{X*}$.
Let $D$ be a nonsingular divisor on $X$ and 
$\iota\colon D\to X$ the embedding morphism.
Then, the following equations hold:
\begin{align}
[\mathcal{O}_X(-D)]=
[\mathcal{O}_X]-[\iota_*\mathcal{O}_{D}], \quad
[\mathcal{O}_X(D)]=
([\mathcal{O}_X]-[\iota_*\mathcal{O}_{D}])^{-1}
=\sum_{k=0}^{\infty}[\iota_*\mathcal{O}_D]^{k}, \\
\chi(X;\mathcal{O}_X(-D))=
\chi([\mathcal{O}_X]-[\iota_*\mathcal{O}_{D}]), \quad
\chi(X;\mathcal{O}_X(D))=
\sum_{k=0}^{\infty}\chi([\iota_*\mathcal{O}_D]^{k}).
\label{equation-power-divisor}
\end{align}
This means that \eqref{equation-power-divisor} is 
the algebraic geometric counterpart of
the first equation in
\cref{conjecture-rr-c-euler,conjecture-rr-euler}.
\end{remark}

The following proposition demonstrates
the relationships among the various conjectures
presented in this paper.

\begin{proposition}
If \label{proposition-euler-to-bertini}
\cite[Conjecture 6.13]{demedrano2023chern},
\cref{conjecture-grr-divisor,conjecture-rr-euler}
are true, then \cref{conjecture-rr-bertini} is also true. 
\end{proposition}

\begin{proof}
We suppose \cref{conjecture-grr-divisor} is true.
From the assumption, 
$D'$ is sedentarity-$0$.

If \cref{conjecture-rr-euler} and 
\cite[Conjecture 6.13]{demedrano2023chern} is true,
then
\begin{align}
\opn{RR}(X;D'-D)=&\opn{RR}(X;D')-
\opn{RR}(D;D'|_{D}) \notag \\
=&\chi(X\setminus D')-\chi(D\setminus (D'\cap D)). \notag
\end{align}
Therefore, the proposition has been proved.
\end{proof}

\begin{remark}
The above proposition suggests that the most difficult part
for proving \cref{conjecture-rr-bertini} is obtaining a proof of
\cite[Conjecture 6.13]{demedrano2023chern}.
\end{remark}

Based on the above, we will generalize
\cref{theorem-rr-euler-surface}.

\begin{theorem}
\label{theorem-rr-bertini-surface}
Let $X$ be a compact tropical surface
and $(D,D')$ a pair of tropical submanifolds of
codimension $1$ in moderate position on $X$.
Then,
\begin{align}
\chi(X\setminus D')-\chi(D\setminus (D'\cap D))
=\frac{\opn{deg}((D'-D).(D'-D-K_X))}{2}+\chi(X).
\end{align}
In particular, \cref{conjecture-rr-bertini} is true 
when $X$ admits a Delzant face structure.
\end{theorem}

\begin{proof}
Since $\dim X=2$, 
\cref{conjecture-grr-divisor} is true from
\cite[Theorem 5.2]{demedrano2023chern}.
From \cref{theorem-rr-euler-surface} and 
\eqref{equation-rr-number-divisor}, we have
\begin{align}
\frac{\opn{deg}((D'-D).(D'-D-K_X))}{2}=&
\opn{RR}(X;D')-\opn{RR}(X;0)-
\opn{RR}(D;D'\cap D) \\
=&\chi(X\setminus D')-\chi(X)-\chi(D\setminus (D'\cap D)).
\end{align}
\end{proof}

\begin{remark}
Let $X$ be an $n$-dimensional compact tropical manifold
and $(D,D')$ a pair of
codimension $1$ tropical submanifold in moderate position
on $X$.
In \cref{remark-c-infinity-divisor},
we stressed
the cohomology $H^{\bullet}(X\setminus D)$ of
the complement $X\setminus D$ is related with
the graded modules $\opn{LMD}^{\bullet}(X;s)$
associated with a permissible
$C^{\infty}$-divisor $s$ on $X$.
In this remark, we explain
$\chi(X\setminus D')-\chi(D\setminus (D'\cap D))$
is also the Euler characteristic of
the relative cohomology 
$H^{\bullet}(X\setminus D,D\setminus D';\mathbb{R})$
of the pair
$(X\setminus D', D\setminus D')$
\cite[Chapter IV. Definition 8.1]{MR842190}
(see also \cite[Chapter II. Proposition 12.3]{MR1481706}).

Let $\mathbb{R}_X$ is the constant sheaf of $\mathbb{R}$
on $X$ and 
$\catn{Mod}(\mathbb{R}_X)$ be the category of sheaves of 
$\mathbb{R}_X$-modules.
Let $Z$ be a locally closed subset of $X$ and
$j\colon Z\to X$ the inclusion map.
For a given $\mathcal{F}\in \catn{Mod}(\mathbb{R}_X)$,
let $(\mathcal{F})_{Z}\deq j_!j^{-1}\mathcal{F}$ 
\cite[Proposition 2.3.6]{MR1299726}.
The functor
$(\cdot)_{Z}\colon \catn{Mod}(\mathbb{R}_X) \to
\catn{Mod}(\mathbb{R}_X);
\mathcal{F}\to (\mathcal{F})_Z$ is exact and has
the right adjoint left exact functor 
$\Gamma_{Z}\colon \catn{Mod}(\mathbb{R}_X) \to
\catn{Mod}(\mathbb{R}_X)$
\cite[Definition 2.3.8]{MR1299726}.
If $Z$ is open, then
$\Gamma_{Z}\mathcal{F}\simeq
j_*j^{-1}\mathcal{F}$
\cite[Proposition 2.3.9 (iii)]{MR1299726}.
Therefore, the following exact sequence exists
\begin{align}
\label{equation-closed-open-exact}
0 \to (\mathbb{R}_X)_{X\setminus D}
\to \mathbb{R}_X \to (\mathbb{R}_X)_{D} \to 0.
\end{align}
The derived functor $R\Gamma_{X\setminus D'}$ also gives
the following exact triangle:
\begin{align}
R\Gamma_{X\setminus D'}((\mathbb{R}_X)_{X\setminus D})
\to  R\Gamma_{X\setminus D'}(\mathbb{R}_X)
\to R\Gamma_{X\setminus D'}((\mathbb{R}_X)_D)
\to R\Gamma_{X\setminus D'}((\mathbb{R}_X)_{X\setminus D})[1].
\end{align}

Let $i\colon D\to X$ be the inclusion map of $D$.
From 
\cite[(2.3.20)]{MR1299726}, we get
\begin{align}
R\Gamma_{X\setminus D'}((\mathbb{R}_X)_D)
&\simeq i_* R\Gamma_{D\setminus D'}\mathbb{R}_D, \\ 
\mathbb{H}^{\bullet}(X;R\Gamma_{X\setminus D'}((\mathbb{R}_X)_{X\setminus D}))
&\simeq H^{\bullet}(X\setminus D';(\mathbb{R}_{X\setminus D'})_{X\setminus (D\cup D')}).
\end{align}
The cohomology $H^{\bullet}(X\setminus D';(\mathbb{R}_{X\setminus D'})_{X\setminus (D\cup D')})$
is just the relative cohomology 
$H^{\bullet}(X\setminus D,D\setminus D';\mathbb{R})$
of the pair
$(X\setminus D', D\setminus D')$.
Hence, the following equation holds
\begin{align}
\chi(\mathbb{H}^{\bullet}(X;R\Gamma_{X\setminus D'}((\mathbb{R}_X)_{X\setminus D})))
&=\chi(X\setminus D')
- \chi(D\setminus (D'\cap D)).
\end{align}
Besides, from definition we can check
\begin{align}
\mathbb{H}^{\bullet}(X;R\Gamma_{X}
((\mathbb{R}_X)_{X\setminus D}))
\simeq H^{\bullet}_c(X\setminus D), \quad
\mathbb{H}^{\bullet}(X;R\Gamma_{X\setminus D'}
((\mathbb{R}_X)_{X}))
\simeq H^{\bullet}(X\setminus D').
\end{align}

Therefore, the graded module 
$\mathbb{H}^{\bullet}(X;R\Gamma_{X\setminus D'}
((\mathbb{R}_X)_{X\setminus D}))$
is a generalization of both
$H^{\bullet}_c(X\setminus D)$ and
$H^{\bullet}(X\setminus D)$, and
$\mathbb{H}^{\bullet}(X;R\Gamma_{X\setminus D'}
((\mathbb{R}_X)_{X\setminus D}))$ is also related
with graded modules associated with
permissible $C^{\infty}$-divisors that are
studied in \cite{tsutsui2023graded}.

We expect that \cref{conjecture-rr-bertini}
is also true for a pair $(D_1,D_2)$ of rational polyhedral
subspaces in moderate position on $X$ such that each $D_i$ is
a finite union of codimension $1$ tropical submanifolds whose
intersections satisfy a good condition.
\end{remark}

\appendix

\section{Compatibility of intersections of tropical cycles}
In this appendix, we note the compatibility of
intersections of tropical cycles in
\cite{MR4637248,demedrano2023chern,shaw2015tropical}.
\begin{proposition}
\label{proposition-two-intersection}
Let $(X,\mathcal{O}_X^{\times})$ be a tropical manifold
and $D_1,D_2\in \opn{Div}^{[0]}(X)$. Then,
the intersection $D_1*D_2$ in the sense of
\cite[\textsection 2.4]{demedrano2023chern} is equal
to $D_1\cdot D_2$ \cite[\textsection 3.4]{MR4637248}.
\end{proposition}
\begin{proof}
\label{proposition-two-pairings}
The definition of
the intersection of Cartier divisors
of both are local.
Therefore, we only need to find an atlas
$\mathcal{U}\deq \{(\psi_i:U_i\to V_i)\}_{i\in I}$
of $X$
such that $U_i$ satisfies
\cref{proposition-two-intersection}.
For any $x\in X$, there exists
a standard chart \cite[Definition 7.2.10]{mikhalkin2018tropical},
i.e., a chart $ \psi_x
\colon U_x \to 
\opn{LC}_x X\times \mathbb{T}^{\opn{sed}_X(x)}$
such that $\psi_x(x)=(0,-\infty)\in
\opn{LC}_x X\times \mathbb{T}^{\opn{sed}_X(x)}$.
Furthermore, $U_x$ can be replaced with the product 
$V_x \times W_x$,
where $V_x$ is a contractible open neighborhood of $0$
in $\opn{LC}_x X$,
and $W_x$ is a contractible open neighborhood of
$-\infty$ in $\mathbb{T}^{\opn{sed}_X(x)}$.
Moreover, we also may assume 
$H^{1}(V_x\times W_x;\mathcal{O}_{V_x\times W_x}^{\times})=0$,
and there exists an open subset $B_x$ of $T_x X$ and
a piecewise integer affine linear function $f$
on $B_x$ such that 
$\opn{div}_{V_x\times W_x}(D_1|_{V_x\times W_x})=\opn{div}_{V_x}(f|_{V_x})\times W_x$.
Therefore, in the context of
\cite{MR4637248}, as well as in the context of
\cite{demedrano2023chern}, the intersection
of sedentarity-0 Cartier divisors at $x$  
can be regarded as the direct product of
the intersection of principal divisors on 
$V_x$ and $W_x$, and thus
it is enough to prove \cref{proposition-two-intersection}
when $\opn{sed}_X(x)=0$. From now on, we assume
$\opn{sed}_X(x)=0$ and $V_x$ is a closed subset of
$B_x$. Let $i \colon V_x\to B_x$ be the closed inclusion
map of $V_x$. Then, for any $D_0\in \opn{Div}(V_x)$ 
\begin{align}
i_*((f|_{V_x})\cdot D_0)=(f)\cdot i_*D_0
\in Z_{n-2}(B_x).
\end{align}
Since $B_x$ is an open subset of a finite dimensional
real vector space, the pairing in the sense of 
\cite[\textsection 3.4]{MR4637248} is
determined by \cite[Definition 3.4]{MR2591823}.
By the same logic, the pairing in the sense
of \cite[\textsection 2.4]{demedrano2023chern}
is determined by the stable intersection of
tropical Cartier divisors and tropical cycles
in $\mathbb{R}^{n}$ \cite[Definition 4.4]{MR2275625}
from \cite[Proposition 2.1.9]{shaw2011tropical}.

As mentioned in \cite[\textsection 2]{MR3032930},
it has been shown in \cite{MR3529085} and \cite{MR2887109}
that the intersection of tropical cycles on
$\mathbb{R}^n$ in the sense of \cite{MR2591823}
coincides with the intersection in the sense of
\cite{MR2149011,MR2275625}.
\end{proof}

Next, we will see the compatibility of
the intersection number of tropical $1$-cycles
in compact surfaces (admitting a global face structure)
in the sense of \cite{shaw2015tropical,demedrano2023chern}
and in our sense. Before proving it,
we recall eigenwave homomorphism from
\cite{MR3961331,MR3894860}.
Every rational polyhedral space $(X,\mathcal{O}_X^{\times})$
has the following exact sequence:
\begin{align}
\label{equation-tropical-exp}
0\to \mathbb{R}_X \to 
\mathcal{O}_X^{\times} \to \Omega_{\mathbb{Z},X}^{1}\to 0.
\end{align}
The exact sequence \eqref{equation-tropical-exp}
defines
the connecting homomorphism
$\phi \colon H^{1,1}(X;\mathbb{Z})
\to H^{0,2}(X;\mathbb{R})$.
The connecting homomorphism
$\phi$ is the dual of the
eigenwave homomorphism
$\hat{\phi} \colon
H_{n-1,n-1}^{\mathrm{BM}}(X;\mathbb{Z})
\to H_{n,n-2}^{\mathrm{BM}}(X;\mathbb{R})$
\cite[(5.2)]{MR3330789}
(see also \cite[Definition 2.9]{MR3894860})
when $X$ is a tropical manifold admitting
a global face structure.
In fact, both of them are compatible with 
the Poincar\'e duality for tropical manifolds
admitting a global face structure
\cite[Lemma 5.13]{MR3894860}.

\begin{proposition}
\label{proposition-cycle-chern}
Let $X$ be a compact tropical surface admitting a global face
structure and $D_1,D_2$
tropical $1$-cycles on $X$. Then,
\begin{align}
\label{equation-two-intersections}
\opn{deg}(D_1 .  D_2)=\int_X \PD{D_1}\cdot \PD{D_2}.
\end{align}
where $D_1. D_2$ is the intersection of $D_1$ and $D_2$
in the sense of \cite{shaw2015tropical}.
\end{proposition}
\begin{proof}
Let $i=1,2$.
From \cite[Theorem 5.4]{MR3330789}
(or more generally \cite[Theorem 1.1]{MR3894860}),
$[D_i]\deq\opn{cyc}_X(D_i)$ is in the kernel 
of the eigenwave homomorphism
$\hat{\phi} \colon
H_{1,1}^{\mathrm{BM}}(X;\mathbb{Z})
\to H_{2,0}^{\mathrm{BM}}(X;\mathbb{R})$
(see also \cite[Theorem 5.13]{MR4637248}
to verify that the cycle map
in \cite[Definition 4.13]{MR3894860}
is equivalent to the cycle map in
\cite[Definition 5.4]{MR4637248}).
Since $\hat{\phi}$ is compatible with
the connecting homomorphism
$\phi \colon H^{1,1}(X;\mathbb{Z})
\to H^{0,2}(X;\mathbb{R})$ \cite[Lemma 5.13]{MR3894860},
we can find $D'_i\in \opn{Div}^{[0]}(X)$
such that 
$[D'_i]=c_1(\mathcal{L}(D'_i))\frown [X]=[D_i]$
from \cite[Proposition 5.12]{MR4637248}.
Then, from 
\cite[Proposition 3.34]{shaw2015tropical} we get
\begin{align}
D_1.D_2=D_1'.D_2', \quad 
\int_X \PD{D_1}\cdot \PD{D_2}
=\int_X \PD{D_1'}\cdot \PD{D_2'}.
\end{align}
Therefore, we may assume $D_1$ and $D_2$ are in
$\opn{Div}(X)^{[0]}$.
Moreover, the definition of trace map,
\cite[Proposition 5.12]{MR4637248}, and the projection
formula deduce
\begin{align}
\int_X \PD{D_1}\cdot \PD{D_2}
&=a_{X*}(c_1(\mathcal{L}(D_1))\frown [D_2])
=a_{X*}(\opn{cyc}_X(D_1\cdot D_2)).
\end{align}
By \cite[Theorem 3.1.7]{shaw2011tropical},
the intersection number of tropical $1$-cycles
in \cite{shaw2015tropical} is equivalent to
that in \cite{shaw2011tropical}.
Furthermore, the intersection of
$D_1,D_2\in \opn{Div}^{[0]}(X)$ defined in
\cite[\textsection 2.4]{demedrano2023chern},
is based on \cite{shaw2011tropical},
so the intersection in 
\cite{demedrano2023chern} and 
that in \cite{shaw2015tropical} are
compatible with each other.
From \cref{proposition-two-intersection},
the intersection of $D_1$ and $D_2$ in
\cite{demedrano2023chern} equivalent to
\cite{MR4637248}, so we obtain
\cref{equation-two-intersections}.
\end{proof}

\bibliography{tropical-complement}

\providecommand{\bysame}{\leavevmode\hbox to3em{\hrulefill}\thinspace}
\providecommand{\MR}{\relax\ifhmode\unskip\space\fi MR }
\providecommand{\MRhref}[2]{%
  \href{http://www.ams.org/mathscinet-getitem?mr=#1}{#2}
}
\providecommand{\href}[2]{#2}
\begin{thebibliography}{LdMRS23}

\bibitem[AB19]{MR3968872}
Karim Adiprasito and Farhad Babaee, \emph{Convexity of complements of tropical
  varieties, and approximations of currents}, Math. Ann. \textbf{373} (2019),
  no.~1-2, 237--251. \MR{3968872}

\bibitem[AC13]{MR3046301}
Omid Amini and Lucia Caporaso, \emph{Riemann-{R}och theory for weighted graphs
  and tropical curves}, Adv. Math. \textbf{240} (2013), 1--23. \MR{3046301}

\bibitem[AEK22]{auroux2022lagrangian}
Denis Auroux, Alexander~I. Efimov, and Ludmil Katzarkov, \emph{Lagrangian floer
  theory for trivalent graphs and homological mirror symmetry for curves},
  2022, arXiv:2017.01981.

\bibitem[Aks23]{MR4626316}
Edvard Aksnes, \emph{Tropical {P}oincar\'{e} duality spaces}, Adv. Geom.
  \textbf{23} (2023), no.~3, 345--370. \MR{4626316}

\bibitem[All12]{MR2961320}
Lars Allermann, \emph{Chern classes of tropical vector bundles}, Ark. Mat.
  \textbf{50} (2012), no.~2, 237--258. \MR{2961320}

\bibitem[AP20]{amini2020hodge}
Omid Amini and Matthieu Piquerez, \emph{Hodge theory for tropical varieties},
  2020, arXiv:2007.07826.

\bibitem[AP21]{amini2021homology}
\bysame, \emph{Homology of tropical fans}, 2021.

\bibitem[AR10]{MR2591823}
Lars Allermann and Johannes Rau, \emph{First steps in tropical intersection
  theory}, Math. Z. \textbf{264} (2010), no.~3, 633--670. \MR{2591823}

\bibitem[Bak08]{MR2448666}
Matthew Baker, \emph{Specialization of linear systems from curves to graphs},
  Algebra Number Theory \textbf{2} (2008), no.~6, 613--653, With an appendix by
  Brian Conrad. \MR{2448666}

\bibitem[BN07]{MR2355607}
Matthew Baker and Serguei Norine, \emph{Riemann-{R}och and {A}bel-{J}acobi
  theory on a finite graph}, Adv. Math. \textbf{215} (2007), no.~2, 766--788.
  \MR{2355607}

\bibitem[Bre97]{MR1481706}
Glen~E. Bredon, \emph{Sheaf theory}, second ed., Graduate Texts in Mathematics,
  vol. 170, Springer-Verlag, New York, 1997. \MR{1481706}

\bibitem[BS15]{MR3339531}
Erwan Brugall\'{e} and Kristin Shaw, \emph{Obstructions to approximating
  tropical curves in surfaces via intersection theory}, Canad. J. Math.
  \textbf{67} (2015), no.~3, 527--572. \MR{3339531}

\bibitem[BU22]{MR4444458}
Madeline Brandt and Martin Ulirsch, \emph{Symmetric powers of algebraic and
  tropical curves: a non-{A}rchimedean perspective}, Trans. Amer. Math. Soc.
  Ser. B \textbf{9} (2022), 586--618. \MR{4444458}

\bibitem[Car15]{cartwright2015combinatorial}
Dustin Cartwright, \emph{Combinatorial tropical surfaces}, 2015,
  arXiv:1506.02023.

\bibitem[Car21]{MR4251610}
Dustin Cartwright, \emph{A specialization inequality for tropical complexes},
  Compos. Math. \textbf{157} (2021), no.~5, 1051--1078. \MR{4251610}

\bibitem[CLS11]{MR2810322}
David~A. Cox, John~B. Little, and Henry~K. Schenck, \emph{Toric varieties},
  Graduate Studies in Mathematics, vol. 124, American Mathematical Society,
  Providence, RI, 2011. \MR{2810322}

\bibitem[Dem70]{MR284446}
Michel Demazure, \emph{Sous-groupes alg\'{e}briques de rang maximum du groupe
  de {C}remona}, Ann. Sci. \'{E}cole Norm. Sup. (4) \textbf{3} (1970),
  507--588. \MR{284446}

\bibitem[FR13]{MR3041763}
Georges Fran\c{c}ois and Johannes Rau, \emph{The diagonal of tropical matroid
  varieties and cycle intersections}, Collect. Math. \textbf{64} (2013), no.~2,
  185--210. \MR{3041763}

\bibitem[Ful98]{MR1644323}
William Fulton, \emph{Intersection theory}, second ed., Ergebnisse der
  Mathematik und ihrer Grenzgebiete. 3. Folge. A Series of Modern Surveys in
  Mathematics [Results in Mathematics and Related Areas. 3rd Series. A Series
  of Modern Surveys in Mathematics], vol.~2, Springer-Verlag, Berlin, 1998.
  \MR{1644323}

\bibitem[GK08]{MR2377750}
Andreas Gathmann and Michael Kerber, \emph{A {R}iemann-{R}och theorem in
  tropical geometry}, Math. Z. \textbf{259} (2008), no.~1, 217--230.
  \MR{2377750}

\bibitem[GS06]{MR2213573}
Mark Gross and Bernd Siebert, \emph{Mirror symmetry via logarithmic
  degeneration data. {I}}, J. Differential Geom. \textbf{72} (2006), no.~2,
  169--338. \MR{2213573}

\bibitem[GS19]{gross2019sheaftheoretic}
Andreas Gross and Farbod Shokrieh, \emph{A sheaf-theoretic approach to tropical
  homology}, 2019, arXiv:1906.08905v1.

\bibitem[GS21]{MR4246795}
Andreas Gross and Farbod Shokrieh, \emph{Cycles, cocycles, and duality on
  tropical manifolds}, Proc. Amer. Math. Soc. \textbf{149} (2021), no.~6,
  2429--2444. \MR{4246795}

\bibitem[GS23]{MR4637248}
\bysame, \emph{A sheaf-theoretic approach to tropical homology}, J. Algebra
  \textbf{635} (2023), 577--641. \MR{4637248}

\bibitem[GUZ22]{MR4512397}
Andreas Gross, Martin Ulirsch, and Dmitry Zakharov, \emph{Principal bundles on
  metric graphs: the {${\rm GL}_n$} case}, Adv. Math. \textbf{411} (2022),
  no.~part A, Paper No. 108775, 45. \MR{4512397}

\bibitem[GW10]{MR2675155}
Ulrich G\"{o}rtz and Torsten Wedhorn, \emph{Algebraic geometry {I}}, Advanced
  Lectures in Mathematics, Vieweg + Teubner, Wiesbaden, 2010, Schemes with
  examples and exercises. \MR{2675155}

\bibitem[GW23]{gortzwedhorn2023}
\bysame, \emph{Algebraic geometry {II}}, Advanced Lectures in Mathematics,
  Springer Spektrum, Wiesbaden, 2023, Cohomology of Schemes.

\bibitem[Hir95]{MR1335917}
Friedrich Hirzebruch, \emph{Topological methods in algebraic geometry},
  Classics in Mathematics, Springer-Verlag, Berlin, 1995, Translated from the
  German and Appendix One by R. L. E. Schwarzenberger, With a preface to the
  third English edition by the author and Schwarzenberger, Appendix Two by A.
  Borel, Reprint of the 1978 edition. \MR{1335917}

\bibitem[IKMZ19]{MR3961331}
Ilia Itenberg, Ludmil Katzarkov, Grigory Mikhalkin, and Ilia Zharkov,
  \emph{Tropical homology}, Math. Ann. \textbf{374} (2019), no.~1-2, 963--1006.
  \MR{3961331}

\bibitem[Ive86]{MR842190}
Birger Iversen, \emph{Cohomology of sheaves}, Universitext, Springer-Verlag,
  Berlin, 1986. \MR{842190}

\bibitem[JMT19]{MR4016643}
Jaiung Jun, Kalina Mincheva, and Jeffrey Tolliver, \emph{Picard groups for
  tropical toric schemes}, Manuscripta Math. \textbf{160} (2019), no.~3-4,
  339--357. \MR{4016643}

\bibitem[JMT24]{MR4646329}
\bysame, \emph{Vector bundles on tropical schemes}, J. Algebra \textbf{637}
  (2024), 1--46. \MR{4646329}

\bibitem[JRS18]{MR3894860}
Philipp Jell, Johannes Rau, and Kristin Shaw, \emph{Lefschetz {$(1,1)$}-theorem
  in tropical geometry}, \'{E}pijournal G\'{e}om. Alg\'{e}brique \textbf{2}
  (2018), Art. 11, 27. \MR{3894860}

\bibitem[JSS19]{MR3903579}
Philipp Jell, Kristin Shaw, and Jascha Smacka, \emph{Superforms, tropical
  cohomology, and {P}oincar\'{e} duality}, Adv. Geom. \textbf{19} (2019),
  no.~1, 101--130. \MR{3903579}

\bibitem[Kaj08]{MR2428356}
Takeshi Kajiwara, \emph{Tropical toric geometry}, Toric topology, Contemp.
  Math., vol. 460, Amer. Math. Soc., Providence, RI, 2008, pp.~197--207.
  \MR{2428356}

\bibitem[Kat12]{MR2887109}
Eric Katz, \emph{Tropical intersection theory from toric varieties}, Collect.
  Math. \textbf{63} (2012), no.~1, 29--44. \MR{2887109}

\bibitem[Kli17]{MR3665000}
Bruno Klingler, \emph{Chern's conjecture for special affine manifolds}, Ann. of
  Math. (2) \textbf{186} (2017), no.~1, 69--95. \MR{3665000}

\bibitem[KS94]{MR1299726}
Masaki Kashiwara and Pierre Schapira, \emph{Sheaves on manifolds}, Grundlehren
  der mathematischen Wissenschaften [Fundamental Principles of Mathematical
  Sciences], vol. 292, Springer-Verlag, Berlin, 1994, With a chapter in French
  by Christian Houzel, Corrected reprint of the 1990 original. \MR{1299726}

\bibitem[KS01]{MR1882331}
Maxim Kontsevich and Yan Soibelman, \emph{Homological mirror symmetry and torus
  fibrations}, Symplectic geometry and mirror symmetry ({S}eoul, 2000), World
  Sci. Publ., River Edge, NJ, 2001, pp.~203--263. \MR{1882331}

\bibitem[KS06]{MR2181810}
\bysame, \emph{Affine structures and non-{A}rchimedean analytic spaces}, The
  unity of mathematics, Progr. Math., vol. 244, Birkh\"{a}user Boston, Boston,
  MA, 2006, pp.~321--385. \MR{2181810}

\bibitem[LdMRS20]{MR3999674}
Luc\'{\i}a L\'{o}pez~de Medrano, Felipe Rinc\'{o}n, and Kristin Shaw,
  \emph{Chern-{S}chwartz-{M}ac{P}herson cycles of matroids}, Proc. Lond. Math.
  Soc. (3) \textbf{120} (2020), no.~1, 1--27. \MR{3999674}

\bibitem[LdMRS23]{demedrano2023chern}
Luc\'{i}a L\'{o}pez~de Medrano, Felipe Rinc\'{o}n, and Kris Shaw, \emph{Chern
  classes of tropical manifolds}, 2023, arXiv:2309.00229.

\bibitem[Mey11]{Meyer2011}
Henning Meyer, \emph{Intersection theory on tropical toric varieties and
  compactifications of tropical parameter spaces}, doctoralthesis, Technische
  Universit{\"a}t Kaiserslautern, 2011.

\bibitem[Mik06]{MR2275625}
Grigory Mikhalkin, \emph{Tropical geometry and its applications}, International
  {C}ongress of {M}athematicians. {V}ol. {II}, Eur. Math. Soc., Z\"{u}rich,
  2006, pp.~827--852. \MR{2275625}

\bibitem[MR18]{mikhalkin2018tropical}
Grigory Mikhalkin and Johannes Rau, \emph{Tropical geometry}, 2018, available
  at \url{https://math.uniandes.edu.co/~j.rau/index_en.html}.

\bibitem[MS74]{MR440554}
John~W. Milnor and James~D. Stasheff, \emph{Characteristic classes}, Annals of
  Mathematics Studies, No. 76, Princeton University Press, Princeton, NJ;
  University of Tokyo Press, Tokyo, 1974. \MR{440554}

\bibitem[MS15]{MR3287221}
Diane Maclagan and Bernd Sturmfels, \emph{Introduction to tropical geometry},
  Graduate Studies in Mathematics, vol. 161, American Mathematical Society,
  Providence, RI, 2015. \MR{3287221}

\bibitem[MZ08]{MR2457739}
Grigory Mikhalkin and Ilia Zharkov, \emph{Tropical curves, their {J}acobians
  and theta functions}, Curves and abelian varieties, Contemp. Math., vol. 465,
  Amer. Math. Soc., Providence, RI, 2008, pp.~203--230. \MR{2457739}

\bibitem[MZ14]{MR3330789}
\bysame, \emph{Tropical eigenwave and intermediate {J}acobians}, Homological
  mirror symmetry and tropical geometry, Lect. Notes Unione Mat. Ital.,
  vol.~15, Springer, Cham, 2014, pp.~309--349. \MR{3330789}

\bibitem[NS16]{MR3498901}
Mounir Nisse and Frank Sottile, \emph{Higher convexity for complements of
  tropical varieties}, Math. Ann. \textbf{365} (2016), no.~1-2, 1--12.
  \MR{3498901}

\bibitem[Ogu18]{MR3838359}
Arthur Ogus, \emph{Lectures on logarithmic algebraic geometry}, Cambridge
  Studies in Advanced Mathematics, vol. 178, Cambridge University Press,
  Cambridge, 2018. \MR{3838359}

\bibitem[Pay09]{MR2511632}
Sam Payne, \emph{Analytification is the limit of all tropicalizations}, Math.
  Res. Lett. \textbf{16} (2009), no.~3, 543--556. \MR{2511632}

\bibitem[Rau16]{MR3529085}
Johannes Rau, \emph{Intersections on tropical moduli spaces}, Rocky Mountain J.
  Math. \textbf{46} (2016), no.~2, 581--662. \MR{3529085}

\bibitem[Rau23]{MR4540954}
\bysame, \emph{The tropical {P}oincar\'{e}-{H}opf theorem}, J. Combin. Theory
  Ser. A \textbf{196} (2023), Paper No. 105733, 28. \MR{4540954}

\bibitem[RGST05]{MR2149011}
J\"{u}rgen Richter-Gebert, Bernd Sturmfels, and Thorsten Theobald, \emph{First
  steps in tropical geometry}, Idempotent mathematics and mathematical physics,
  Contemp. Math., vol. 377, Amer. Math. Soc., Providence, RI, 2005,
  pp.~289--317. \MR{2149011}

\bibitem[RS82]{MR665919}
Colin~Patrick Rourke and Brian~Joseph Sanderson, \emph{Introduction to
  piecewise-linear topology}, Springer Study Edition, Springer-Verlag,
  Berlin-New York, 1982, Reprint. \MR{665919}

\bibitem[Rud21]{MR4347312}
Helge Ruddat, \emph{A homology theory for tropical cycles on integral affine
  manifolds and a perfect pairing}, Geom. Topol. \textbf{25} (2021), no.~6,
  3079--3132. \MR{4347312}

\bibitem[Sch91]{MR1115569}
Pierre Schapira, \emph{Operations on constructible functions}, J. Pure Appl.
  Algebra \textbf{72} (1991), no.~1, 83--93. \MR{1115569}

\bibitem[Sha11]{shaw2011tropical}
Kristin Shaw, \emph{Tropical intersection theory and surfaces}, Ph.D. thesis,
  2011.

\bibitem[Sha13]{MR3032930}
Kristin~M. Shaw, \emph{A tropical intersection product in matroidal fans}, SIAM
  J. Discrete Math. \textbf{27} (2013), no.~1, 459--491. \MR{3032930}

\bibitem[Sha15]{shaw2015tropical}
Kristin Shaw, \emph{Tropical surfaces}, 2015, arXiv: 1506.07407.

\bibitem[Sma17]{smacka2017differential}
Jascha Smacka, \emph{Differential forms on tropical spaces}, Ph.D. thesis,
  2017.

\bibitem[Sum74]{MR337963}
Hideyasu Sumihiro, \emph{Equivariant completion}, J. Math. Kyoto Univ.
  \textbf{14} (1974), 1--28. \MR{337963}

\bibitem[Sum21]{MR4229604}
Ken Sumi, \emph{Tropical theta functions and {R}iemann-{R}och inequality for
  tropical {A}belian surfaces}, Math. Z. \textbf{297} (2021), no.~3-4,
  1329--1351. \MR{4229604}

\bibitem[SW22]{shaw2022birational}
Kris Shaw and Annette Werner, \emph{On the birational geometry of matroids},
  2022, arXiv:2207.13639.

\bibitem[Tsu23]{tsutsui2023graded}
Yuki Tsutsui, \emph{Graded modules associated with permissible
  {$C^{\infty}$}-divisors on tropical manifolds}, 2023, arXiv:2306.08905.

\bibitem[Vig10]{MR2594592}
Magnus~Dehli Vigeland, \emph{Smooth tropical surfaces with infinitely many
  tropical lines}, Ark. Mat. \textbf{48} (2010), no.~1, 177--206. \MR{2594592}

\bibitem[Vir88]{MR970076}
O.~Ya. Viro, \emph{Some integral calculus based on {E}uler characteristic},
  Topology and geometry---{R}ohlin {S}eminar, Lecture Notes in Math., vol.
  1346, Springer, Berlin, 1988, pp.~127--138. \MR{970076}

\bibitem[Yam21]{yamamoto2021tropical}
Yuto Yamamoto, \emph{Tropical contractions to integral affine manifolds with
  singularities}, 2021.

\end{thebibliography}
\bibliographystyle{amsalpha}

\end{document}